\title{Universality results for random matrices over finite local rings}
\author{Nikita Lvov\footnote{nikita.lvov@mail.mcgill.ca}}

\documentclass{article}
\usepackage{amsmath}
     \usepackage{
     hyperref,
     cleveref,
     }

\newcounter{Chapcounter}

\newcommand{\chapter}[1] 
{ {\centering          
  \addtocounter{Chapcounter}{1} \huge \textbf{ \color{black} Chapter \theChapcounter: ~#1} }   
  \addcontentsline{toc}{section}{ \color{black} Chapter:~\theChapcounter~~ #1}    
}

\hypersetup{colorlinks=true}
\hypersetup{linkcolor=black}

\usepackage{relsize}

\usepackage{ifxetex,ifluatex}
\if\ifxetex T\else\ifluatex T\else F\fi\fi T%
  \usepackage{fontspec}
\else
  \usepackage[T1]{fontenc}

  \usepackage{blkarray, bigstrut}
  \usepackage[utf8]{inputenc}
  
  \usepackage{amsthm}
  \usepackage{thmtools}
  \usepackage{mathtools}

  \usepackage{lmodern}
  \usepackage{amssymb}
  \usepackage{amsfonts}
  \usepackage{tikz-cd}
  \usepackage{amsthm}
  \usepackage{xfrac}
  \usepackage{mathtools}
  \usepackage{multirow}
  \usepackage{mathrsfs}
  \usepackage{comment}
  \newtheorem{theorem}{Theorem}[section]
  \newtheorem{lemma}[theorem]{Lemma}
  \newtheorem{prop}[theorem]{Proposition}
  \newtheorem*{claim}{Claim}

  \newtheorem{sublemma}{Lemma}[lemma]
  \newtheorem*{corollary}{Corollary}

  \newcommand{\probP}{\text{I\kern-0.15em P}}
  \newcommand{\probE}{\text{I\kern-0.15em E}}
  \newcommand{\coker}{\text{coker}}
  \newcommand{\defeq}{\overset{\mathrm{def}}{=\joinrel=}}

  \numberwithin{equation}{section}

  \theoremstyle{remark}
  \newtheorem*{remark}{Remark}
  \theoremstyle{remark}

  \newtheorem*{definition}{Definition}

  \newcommand{\rows}{l}
  \newcommand{\cols}{k}

  \newcommand{\Q}{\mathbb{Q}}
  \newcommand{\Z}{\mathbb{Z}}

  \newcommand{\Hom}{\text{Hom}}
  \excludecomment{mysection}
  \excludecomment{mymysection}
  \excludecomment{maybeinclude}

  \newcommand{\supp}{\pi}

    \newenvironment{mat}
    {
    \left[
\begin{array}}
{ \end{array}
\right]
}

\newcommand{\probM}{\mathcal{M}}
\newcommand{\probU}{\mathcal{U}}

\newcommand{\R}{\mathbb{R}}

\newcommand{\im}{\text{im}}

\usepackage{mdframed}

\newcommand{\vvector}{v_2}
\newcommand{\uvector}{v_1}

\newtheorem*{theorem*}{Theorem}
\newenvironment{ftheo*}
  {\begin{mdframed}\begin{theorem*}}
  {\end{theorem*}\end{mdframed}}

\newcommand{\ssfrac}[2]{#1 \Big/ #2}

\newtheorem{Corollary}{Corollary}[theorem]

\newtheorem*{goal}{Goal}

  \usepackage{enumitem}

   \DeclareSymbolFont{bbold}{U}{bbold}{m}{n}
   \DeclareSymbolFontAlphabet{\mathbbold}{bbold}
   \newcommand{\one}{\mathbbold{1}}
   
   \theoremstyle{remark}

\fi

 \theoremstyle{definition}

 \newcommand{\comm}[1]{} 

 \newcommand{\lone}{L^1(X_0)}

 \newcommand{\prob}{\probP}

 \usepackage{cite}
 \newcommand{\commm}[1]{}
 \newcommand{\marginparr}{\commm}

 \newcommand{\sspace}{\newline
 \newline
 \noindent}
\renewcommand{\rows}{n}
\renewcommand{\cols}{n+u}
\newcommand{\Mm}{M}
\newcommand{\Mmodule}{\mathbf{M}}
\newcommand{\projj}{\text{proj}}
\newcommand{\One}{\mathlarger{\mathlarger{\one}}}
\newcommand{\probMM}{
\mathcal{M}_{\rows,\cols}
}
\newcommand{\probMMM}{\mathcal{N}}

\newcommand{\linf}[1]{ \Big| \Big| #1 \Big| \Big|_{l^\infty} }

\renewcommand{\lone}[1]{ \Big| \Big| #1 \Big| \Big|_{l^1} }

\newcommand{\ltwo}[1]{
\Big| \Big| #1 \Big| \Big|_{l^2}
}

\newcommand{\largematrix}[1]{
\begin{mat}{ccc|ccc}
&&&*&\hdots&\\
&#1&&\vdots&&\\
&&&*&\hdots&\\ \hline
*&&*&*&\hdots&\\
\vdots&&\vdots&\vdots&\ddots&\\
&&&&&
\end{mat}
}

\renewcommand{\lone}[1]{
\Big| \Big| #1 \Big| \Big|_{l^1}
}

\newcommand{\maxmod}{
\mathfrak{m}
}

\newcommand{\maxbound}
{
\max
\left[
\linf{\xi \mod \maxmod}
\,
,
\,
\frac{1}{ char(R/\maxmod) }
\right]
}

\renewcommand{\marginparr}{\commm}

\begin{document}


\maketitle

\abstract{Let $R$ be a finite local ring. We prove a quantitative
universality statement for the cokernel of random matrices with i.i.d.
entries valued in $R$.
Rather than use the moment method, we use the Lindeberg replacement
technique. This approach also yields a universality result for several
invariants that are finer than the cokernel, such as the
determinant.}

\tableofcontents
\newpage
\section{Overview}

Cokernels of random matrices over $\Z_p$ have recently been a topic of active research. Historically the first random model to be studied was the Haar random matrix model \cite{FriedmanWashington1989}. In this case, the distribution of the cokernel can be computed explicitly, for any matrix dimension. The reason for this is the large amount of symmetry possessed by the Haar model.

\paragraph{} Subsequently, in the work of Wood and others (\cite{Maples1}, \cite{WoodIntegral}, \cite{WoodNguyen} ), it was found that the distribution of the cokernel of any large square i.i.d. matrix asymptotically has the same distribution as in the Haar case, provided that a necessary condition is met - the distribution of the entries must not be supported on the translate of the maximal ideal of $\Z_p$. This is the phenomenon of universality for random $p$-adic matrices. 

\paragraph{}The purpose of this paper is to prove universality for i.i.d. matrices over any finite local ring $R$, under a slightly stronger condition - the distribution of the entries must not be supported on the translate of the maximal ideal of $R$, or the translate of a subring of $R$. The technique we use is the Lindeberg Replacement Strategy. This is allows to establish an exponential convergence rate for the total variation distance between the i.i.d. case and the Haar case. Furthermore, this technique allows us to prove universality for invariants other than the cokernel, such as the determinant and the span of the column vectors.

\paragraph{}During the preparation of this manuscript, a related work by J. Shen \cite{JiaheShen} appeared on the arXiv. Shen independently developed a proof methodology similar to the one presented here. This methodology is applied in \cite{JiaheShen} to establish quantitative universality for the distribution of cokernels of symmetric and anti-symmetric $p$-adic random matrices. Shen's results for the symmetric case are particularly striking, and are beyond the reach of the methods developed in the present paper.

\subsection{Introduction}

We will now go into more detail. Firstly, we discuss Haar random matrices over $\Z_p$, that is matrices whose entries are uniformly distributed and independent elements.

\paragraph{Matrices with Haar random entries}
\begin{enumerate}
\item[] \textit{Square matrices over $\Z_p$.} First, let $\mathcal{U}_{n,m}$ be an $n \times m$ matrix over
$\Z_p$, whose
entries are sampled at random from the Haar measure. A theorem of Friedman and
Washington describes the asymptotic distribution of
$coker(\mathcal{U}_{n,n})$:

\begin{theorem*} \cite[Proposition 1]{FriedmanWashington1989}
\label{thm:friedmanwashington}
\begin{equation}
\label{eqn:friedmanwashington}
\lim_{n \rightarrow \infty} \probP\Big( coker(\mathcal{U}_{n,n}) \cong A \Big)
=
\frac{c_0}{|Aut(A)|}
\end{equation}
where
\[
c_0=\prod_{i=1}^{\infty} \left( 1 - \frac{1}{p^{i}}  \right)
\]
\end{theorem*}

The distribution (\ref{eqn:friedmanwashington}) on $p$-groups is known
as the Cohen-Lenstra distribution \cite{CohenLenstra}.

\item[] \textit{Rectangular matrices over $\Z_p$.} The derivation that establishes \cite[Proposition 1]{FriedmanWashington1989} can be generalized to non-square matrices to give:
\begin{equation}
\label{eqn:friedmanwashingtonu}
\lim_{n \rightarrow \infty} \probP
\Big(
coker(\mathcal{U}_{n,n+u})
\cong
A
\Big)
=
\frac{c_u}{|A|^{u} |Aut(A)|}
\hspace{0.5in} \text{ for $u \geq 0$}
\end{equation}
where
\[
c_u = \prod_{i=u+1}^{\infty} \left( 1 - \frac{1}{p^i} \right)
\]
\item[] \textit{Rectangular matrices over a finite local ring $R$.} Finally, from the recent work of Sawin and Wood
\cite[Lemma 6.7 and Lemma 6.6]{SawinWood}, we can deduce a formula valid for any
\textit{finite} local ring $R$. We consider an $n \times (n+u)$ matrix over $R$,
whose entries are independent and uniformly distributed. We again
denote this matrix as $\probU_{n,n+u}$. \cite[Lemma 6.7]{SawinWood} implies
that for $u>0$, and any finite local ring $R$,
\[
\lim_{n \rightarrow \infty} \probP(coker(\probU_{n,n+u})=A)
=
\]
\begin{equation}
\label{eqn:sawinwood}
\frac{1}{|A|^u |Aut(A)|}
\prod^{\infty}_{i=d(A)+u+1} \left(1 - \frac{1}{q^i} \right)
\end{equation}
where $q$ is the cardinality the residue field of $R$.
$d(A)$ is defined to be the difference between the number of relations
and the number of elements in the minimal presentation of
$A$, negative if there are more relations than
elements\footnote{For example, $d(R^3)=3$.}.
\newline
\newline
In fact, (\ref{eqn:friedmanwashington}) and (\ref{eqn:friedmanwashingtonu}) can both be deduced from (\ref{eqn:sawinwood}).
\end{enumerate}

\subsubsection{Universality - the main results of this paper}

Firstly, we recall the statement that we wish to generalize. 
As found in the work of Maples, Wood and Nguyen (\cite{Maples1},
\cite{WoodIntegral}, \cite{WoodNguyen}), the conclusion of
(\ref{eqn:friedmanwashington}) continues to hold when we replace the
matrix $\mathcal{U}_{n,n}$ by \textit{any} i.i.d. random matrix, under the
necessary condition that the distribution
of the entries
is non-degenerate modulo $p$. This is an instance of the general
phenomenon of universality. We refer to \cite{WoodICM} for a survey. 
\newline
\newline
Henceforth, $R$ will be a fixed finite local ring, and $u$ will be
an integer. The integer $u$ may be negative.
\newline
\newline
In this paper, we consider a $n \times (n+u)$ random matrix
$\mathcal{M}_{n,n+u}$ over $R$. We assume that
the entries are i.i.d. random variables. We assume that their
distribution is not
concentrated on the translate of a subring, or the translate of an
ideal of $R$.
\newline
\newline
We prove an estimate, which we call the column-swapping
estimate (\autoref{thm:quantheorem}). This estimate will immediately imply the following universality theorem:

\marginparr{Maybe include the explicit statement of universality for
the cokernel}

\begin{theorem}
\label{thm:firsttheorem:estimate}
Let $\probM_{n,n+u}$ and $\probU_{n,n+u}$ be the random matrices previously defined. We have the following asymptotic statements:
\begin{itemize}
\item[(A)] For any $u \in \Z$, the total variation distance between
\[
coker(\probM_{n,n+u}) \text{ and } coker(\mathcal{U}_{n,n+u})
\]
tends to $0$ as $n \rightarrow \infty$.
\item[(B)] The total variation distance between the joint distribution of
\[
   coker(\probM_{n,n}) \, , \, det(\probM_{n,n})
\]
and the joint distribution of
\[
coker(\mathcal{U}_{n,n}) \, , \,  det(\mathcal{U}_{n,n})
\]
tends to $0$ as $n \rightarrow \infty$.
\item[(C)] Both of the above statements hold if we replace the
cokernel by the span of the column vectors.
\end{itemize}

In all of the above cases, the total variation distance is bounded above by $C  \theta^n$, for any $\theta$ satisfying the inequality (\ref{eqn:thetaequation}). The constant $C$ depends on $\theta$, $u$, $R$ and the distribution of the entries of $\probM_{n,n+u}$. 
\end{theorem}

Statement (A) has the following corollary:

\begin{corollary}
When $u\geq 0$, and $A$ is any $R$-module, the asymptotic value of
$\probP(coker(\probM_{n,n+u}) \cong A)$ is given by (\ref{eqn:sawinwood}).
\end{corollary}

Statement (B) has the following corollary:
\begin{corollary}
$det(\probM_{n,n})$ has the same asymptotic distribution as $det(\probU_{n,n})$.
\end{corollary}

\subsubsection{Method of proof}

Most current proofs of universality for random matrices, over finite
and pro-finite rings, use the moment method, which first appeared in
\cite{WoodMoments}. We use a different approach, the Lindeberg
replacement technique of \cite{TaoVuTwo}, inspired by \cite{Lindeberg}. The idea of this technique
is to replace a column vector of $\probM_{n,n+u}$ by a uniformly
random vector. Somewhat surprisingly, this does not significantly
alter the distribution of the cokernel of the random matrix or the
distribution of the other invariants  (\autoref{thm:qualtheorem}).
Replacing all the columns by independent uniformly random vectors allows us to conclude \ref{thm:firsttheorem:estimate}.
\newline
\begin{remark}
In an upcoming note, we will reinterpret the "column-swapping estimate" from a dynamical point of view - the estimate will imply
that cokernels of minors \textit{approximately} form a Markov chain.
We will see that the dynamic perspective naturally allows us to prove finer universality results, using equidistribution theorems for Markov chains.
\end{remark}

\subsubsection{Related work}

The earliest papers treating universality problems for random matrices
over the $p$-adics were \cite{Maples1} , \cite{WoodMoments} and
\cite{WoodNguyen} . In the decade that followed their appearance, there
has been a surge of results on this topic. We again refer to
\cite{WoodICM} for a survey. \marginparr{Mention some results that have
appeared since 2022... CheongYu, AvW...} Nearly all of this work has
been driven by the moment  method, introduced by Wood in the seminal
article on symmetric $p$-adic matrices \cite{WoodMoments}.
\newline
\newline
As noted previously, we instead use the Lindeberg replacement technique. In
the context of random matrices over the real and complex numbers, this
method was introduced by Tao and Vu in \cite{TaoVuTwo}; in the latter article, it is used to
prove universality of local eigenvalue statistics of random matrices over $\mathbb{R}$ and $\mathbb{C}$.
The LRT has also been used to prove universality for the distribution
of the logarithm of the determinant of a random matrix over
$\mathbb{R}$\marginparr{pver $\R$???} in \cite{NguyenVu}.
\newline
\newline
For rings other than $\mathbb{R}$ or $\mathbb{C}$, this
approach seems to have been pursued only in the case of finite fields.
In particular, for symmetric matrices over finite fields, a strategy
equivalent to a replacement strategy was suggested in an unpublished
note of Maples \cite{Maples2}\marginparr{see also KoymansPagano for a similar idea}.

\subsubsection{Outline of the paper}

In \S \ref{sec:estimate}, we state the column-swapping estimate and
derive some of its consequences, such as
\autoref{thm:firsttheorem:estimate}. In \S \ref{sec:proofofestimate:main}, we
reduce the proof of the column-swapping estimate to two
inequalities, proven in \S \ref{sec: random matrix inequality} as \autoref{thm:fundamentalmomentinequality} and in \S \ref{sec: decomposing measures} as \autoref{thm:maininequality}.
\newline
\newline
\S \ref{sec: random matrix inequality} and \S \ref{sec: decomposing measures} contain the auxiliary results that are needed to prove the column-swapping estimate. We remark that these two sections are structured in such a way that each may be read independently of the rest of the paper.

\subsubsection{Acknowledgments} 
  
This project is the culmination of many years of work, and many people have contributed along the way. First of all, the author would like to than his advisor Manjul Bhargava for the initial suggestion that he think about universality for p-adic matrices, and for encouragement and inspiration. The author would like to thank Yakov Sinai, who contributed the original idea that led to this line of attack, and that inspired the whole Markovian approach to random matrices. Moreover, the author would like to thank Prof. Sinai for his constant encouragement, and the opportunity to speak at his seminar. The author would like to thank Xi Sisi Shen for her invaluable support during the difficult initial stages of this project. The author would like to thank Timothy Hurd for his faith and friendship, at a time when the author had nearly given up on mathematics. The author would like to thank his family for their constant love and support, without which it would have been tremendously difficult to continue, let alone complete this work. The author would also like to thank L.V. Chebotarev, Matthew Egan, James Gutteridge and Dima and Mahshid Grudin for their interest, kind encouragement and helpful advice. The author would like to thank Alexander van Werde for useful discussions concerning universality. Finally, Alex Yu provided the lion's share of feedback, advice, constructive criticism, and good humour that the author benefited from during the writing of this paper.

\section{The column-swapping estimate and its consequences}
\label{sec:estimate}
\renewcommand{\vvector}{\mathbf{v_1}}
\newcommand{\vvvector}{\mathbf{v}}
\renewcommand{\uvector}{\mathbf{v_0}}

\subsection{Preliminary definitions}

\begin{definition}
Let $\xi$ be a random variable on $R$ with the same distribution as the entries of $\probM_{n,n+u}$.
\end{definition}

\begin{definition}
Let $\vvector$ be a random vector in $R^n$ whose entries are independent and have the same distribution as $\xi$.
\end{definition}

\begin{definition}
Let $\uvector$ be a uniformly distributed random element of $R^n$.
\end{definition}

\newcommand{\projjj}[2]{
{#2} \,+ {#1}
}

\begin{definition}
Suppose that $B$ is a submodule of the module $R^n$ and that
$\mathbf{v}$ is a random element of $R^n$. Then by
\[
\projjj{B}{\vvvector}
\]
we denote the sum of the random vector $\vvvector$ and an independent uniformly
distributed random
element of $B$.
\end{definition}

\begin{remark}
Intuitively, we get the distribution of $\projjj{B}{\vvvector}$ by "averaging out" the distribution of $\vvvector$
over $B$-cosets.
\end{remark}

\newcommand{\matrixvectortwo}[2]
{
\begin{mat}{c|c}
#1 & #2 \\
\end{mat}
}

\newcommand{\matrixvector}[2]{
\begin{mat}{ccc|c}
&&&\\
&#1&&#2\\
&&&\\
\end{mat}
}

Observe that, by the invariance of the cokernel under the action of
$SL(R)$, the distribution of
\[
\coker
\matrixvector{ \probM_{n,n+u}}{\vvvector}
\]
remains invariant if we add a uniformly random element of
$im(\probM_{n,n+u})$ to $\vvvector$.
\newline
\newline

\newcommand{\imM}{im(\probM_{n,n+u})}

\subsection{Formulation of the column-swapping estimate}

It follows from the preceding paragraph that the total variation distance:
\begin{equation}
\label{eqn:totalvariationdistance}
d_{TV} \left(
coker
\matrixvector{\probM_{n,n+u}}{\vvector}
,
coker
\matrixvector{\probM_{n,n+u}}{\uvector}
\right)
\end{equation}
is equal to the total variation distance between
\[
coker
\matrixvectortwo{\probM_{n,n+u}}{\projjj{im(\probM_{n,n+u})}{\vvector}}
\]
and
\[
coker
\matrixvectortwo{\probM_{n,n+u}}{\projjj{im(\probM_{n,n+u})}{\uvector}}
\]
Therefore, (\ref{eqn:totalvariationdistance}) is bounded by:
\begin{equation}
\label{eqn:projMuv}
\sum_{M}
\probP(\probM_{n,n+u} = M )
d_{TV} \Big(
\projjj{im(M)}{\vvector}
,
\projjj{im(M)}{\uvector}
\Big)
\end{equation}

\begin{remark}
Of course, $\projjj{im(M)}{\uvector}$ has the same
distribution as $\uvector$. Therefore, we can replace it by
$\uvector$, if we wish.
\end{remark}

It therefore suffices to bound (\ref{eqn:projMuv}). The preceding discussion
serves to motivate the following theorem, which is the central result
of this paper: we first state the theorem in qualitative form:

\begin{theorem}[Qualitative form of the column-swapping estimate]
\label{thm:qualtheorem}
Suppose that $\probM_{n,n+u}$ is a random matrix with i.i.d. random entries,
which are sampled from a distribution that is
\begin{itemize}
\item not concentrated on the translate of an ideal of $R$,
\item not concentrated on the translate of a subring of $R$.
\end{itemize}
Then there exists $\theta<1$ such that
\[
\sum_M \probP(\probM_{n,n+u} = M) d_{TV} \Big(
\projjj{im(M)}{\vvector}
,
\projjj{im(M)}{\uvector}
\Big)
\leq
\]
\begin{equation}
\label{eqn:lrtinequality}
\leq O(\theta^n)
\end{equation}
where $\theta$ is a constant that depends on $R$, and on the distribution
of the entries of $\probM_{n,n+u}$.
\newline
\newline
Moreover, the same estimate continues to hold if some of the entries
of $\probM_{n,n+u}$ are replaced by independent uniformly distributed
random variables.
\end{theorem}

\paragraph{Quantitative form}

Below, we give the quantitative form of \autoref{thm:qualtheorem}. In order to do so, we must first introduce some definitions. Let
$\xi$ be a random variable that has the same \textit{distribution} as the
entries of $\probM_{n,n+u}$.
\newline
\begin{definition}
Define $(\xi \mod \mathfrak{m})$ to be the random variable that is induced by
$\xi$ on $R/\mathfrak{m}$.
\end{definition}

\begin{itemize}
\item Denote by $\ltwo{\xi \mod \mathfrak{m}}$ the $l^2$ norm of the
distribution of $(\xi \mod \mathfrak{m})$.
\item Denote by $\linf{\xi \mod \mathfrak{m}}$ the $l^{\infty}$ norm
of the distribution of $(\xi \mod \mathfrak{m})$.
\item Let $char(R / \mathfrak{m})$ denote the characteristic of the
field $R/\mathfrak{m}$.
\end{itemize}

\newcommand{\maxlnorm}{
\max \Big(
\ltwo{\xi \mod \mathfrak{m}},
\linf{\xi \mod \mathfrak{m}},
\frac{1}{char(R/\mathfrak{m})}
\Big)
}

\begin{remark}
Observe that, under the hypotheses of \autoref{thm:qualtheorem},
\[
\maxlnorm < 1.
\]
\end{remark}

\begin{theorem}[Quantitative form of the column-swapping estimate]
\label{thm:quantheorem}
If $\theta$ satisfies
\begin{equation}
\label{eqn:thetaequation}
\maxlnorm < \theta < 1
\end{equation}
then, (\ref{eqn:lrtinequality}) holds; the proportionality constant
implicit in
(\ref{eqn:lrtinequality}) depends on $u$, $\theta$, $R$ and on the
smallest non-zero value of:
\[
\probP(\xi = r) \hspace{0.1in} , \hspace{0.1in} r \in R
\]
\end{theorem}

\paragraph{Intuitive meaning of \autoref{thm:qualtheorem}}

\autoref{thm:qualtheorem} captures the following intuitively plausible fact. An i.i.d. random vector should be approximately equidistributed in the quotient of $R^n$ by $n+u$ other i.i.d. random vectors, with high probability.

\subsubsection{Corollaries of \autoref{thm:quantheorem}}
\label{sec:proofofestimate}
\paragraph{Deduction of \autoref{thm:firsttheorem:estimate} }

\autoref{thm:firsttheorem:estimate} is an immediate
consequence of \autoref{thm:quantheorem}.
\begin{proof} Indeed, by the discussion preceding
\autoref{thm:qualtheorem}, the
inequality (\ref{eqn:lrtinequality}) implies that:
\[
d_{TV} \left( coker
\begin{mat}{ccc|c}
&&& \\
&\probM_{n,n+u}&& \vvector \\
&&&
\end{mat}
,
coker
\begin{mat}{ccc|c}
&&& \\
&\probM_{n,n+u}&& \uvector \\
&&&
\end{mat}
  \right) \leq
\]
\begin{equation}
\label{eqn:cokerneltv}
\leq
O(\theta^n)
\end{equation}

\begin{remark}
Recall that the inequality in (\ref{eqn:lrtinequality}) remains valid if we replace some of the entries of $\probM_{n,n+u}$ by independent,
uniformly random variables. Hence, (\ref{eqn:cokerneltv}) also remains
valid.
\end{remark}

By the preceding remark, we can apply the inequality (\ref{eqn:cokerneltv}) iteratively
to conclude that
\begin{equation}
\label{eqn:cokerneluniversality}
\lim_{n \rightarrow \infty}
d_{TV} \Big(
coker
(\probM_{n,n+u+1}),
coker
(\probU_{n,n+u+1})
\Big)=0
\end{equation}
for any $u \in \Z$.
\end{proof}

\paragraph{Finer invariants}

To conclude (\ref{eqn:cokerneltv}), all we have used about the
cokernel function is its invariance under the right action of $SL_{n+u}$.
Therefore, (\ref{eqn:cokerneltv}) remains valid if we replace the
cokernel by any other $SL_{n+u}$ invariant, such as the span of the column vectors of $\probM_{n,n+u}$. In particular,
\begin{itemize}
\item We have
\[
\lim_{n \rightarrow \infty}
d_{TV} \Big(
span(\probM_{n,n+u})
,
span(\probU_{n,n+u})
\Big) =0
\]
\item The total variation distance between the joint distribution of
\[
span(\probM_{n,n}) \, , \, det(\probM_{n,n})
\]
and the joint distribution of
\[
span(\probU_{n,n}) \, , \, det(\probM_{n,n})
\]
tends to $0$.
\end{itemize}

\paragraph{Rate of convergence}

We observe that the rate of convergence is 
\[
(n+u)O(\theta^{n})
\]
which may be rewritten as 
\[
O(\theta^n)
\]
for a slightly larger $\theta$ in the range (\ref{eqn:thetaequation}). Once again, the implicit constant depends only on $u$, $R$, $\theta$ and the smallest non-zero value of 
\[
\probP(\xi = r) \hspace{0.2in} r \in R
\]

This proves the rest of \autoref{thm:firsttheorem:estimate}

\paragraph{Upper-Triangular Matrices}

We will finally observe a consequence of \autoref{thm:quantheorem} that
will be useful in an upcoming note, where \autoref{thm:quantheorem} will be combined with a Markovian perspective to deduce more universality results.
\newline
\newline

\begin{definition}
Define $T_n$ to be the group of upper triangular matrices with $1$'s
on the diagonal and let $t$ be the map to the double quotient
\[
Mat_{k,l} \rightarrow T_k \backslash Mat_{k,l} / T_l
\]
\end{definition}

\begin{corollary} (of \autoref{thm:quantheorem})
\[
d_{TV}\left( t \matrixvector{\probM_{n,n+u}}{\vvector} ,
t \matrixvector{\probM_{n,n+u}}{\uvector}  \right) \leq O(\theta^n)
\]
\end{corollary}

\marginparr{The intuitively plausible fact that an i.i.d. random vector
should be approximately equidistributed in the quotient of $R^n$ by
the other column vectors
of the matrix. Though intuitively plausible, the proof of this
statement is the main
technical ingredient of this paper...}


\subsection{Proof of the column-swapping estimate, using the results of \S \ref{sec: random matrix inequality} and \S \ref{sec: decomposing measures}}
\label{sec:proofofestimate:main}

The purpose of this subsection is to reduce the column-swapping estimate to two inequalities, proven in \S \ref{sec: random matrix inequality} and \S \ref{sec: decomposing measures}, respectively.
\newline
\newline
First, given a signed measure $\nu$ on a finite module $\Mmodule$, and a submodule $N \subset \Mmodule$, we denote by $\projj_N \nu $ the average of $\nu$ over $N$-cosets of $\Mmodule$.
\newline
\newline
We wish to bound

\begin{equation}
\label{eqn:tvineq}
\sum_{M}\probP(\probM_{n,n+u}=M) d_{TV} \Big(
\projjj{im(M)}{\vvector}
,
\projjj{im(M)}{\uvector}
\Big)
\end{equation}

\paragraph{}Let $\nu_i$ denote the distribution of $\mathbf{v}_i$. Note that
\[
d_{TV} \Big(
\projjj{im(M)}{\vvector}
,
\projjj{im(M)}{\uvector}
\Big)
\]
can be rewritten as
\[
\lone{\projj_{im(M)} (\nu_1 - \nu_0) }
\]

Hence, to prove \autoref{thm:quantheorem}, it suffices to bound:
\[
\sum_M \probP(\probM_{n,n+u} = M) \lone{\projj_{im(M)} (\nu_1 - \nu_0)}
\]
We will now see that \autoref{thm:quantheorem} can be deduced from the following estimate:

\begin{lemma}
\label{lem:centralinequality}
For any arbitrary signed measure $\nu$ on $R^n$, we have
the inequality:
\[
\sum_{M}
\probP
\left(
\probM_{\rows,\cols} = M
\right)
\lone{
\projj_{im(M)}
\nu 
}
\leq
\]
\begin{equation}
\label{eqn:centralinequality}
\leq
\sum_{I \subset R}
O
\left( (1+\epsilon)^{\rows} 
\ltwo{
\nu  \mod I
}
\right)
     +
\end{equation}
\[
O
\left(
(1+\epsilon)^{\rows}
\maxbound^{\rows} 
\right)
\]
\marginparr{There is an unclear comment in your notes.}
where the implied constants depend on $R$, $u$, $\epsilon$, and the
minimal non-zero value of  $
\probP(\xi=r) \text{, for } r \in R$.
\end{lemma}

\paragraph{The deduction of \autoref{thm:quantheorem} from \autoref{lem:centralinequality}}
In order to deduce \autoref{thm:quantheorem}, we will apply the bound in \autoref{lem:centralinequality} to the case when
\[
\nu = \nu_1 - \nu_0
\]
First of all, note that
\begin{itemize}
\item The total mass of the measure $\nu_1 - \nu_0$ is $0$:
\begin{equation}
\label{eqn:chizero} 
\nu_1 - \nu_0 \mod R = 0
\end{equation}
\item Hence, $\nu_1 - \nu_0$ is orthogonal to $\nu_0$.
\end{itemize}
It follows that 
\[
\ltwo{\nu \mod I} = \ltwo{\nu_1 - \nu_0 \mod I} \leq \ltwo{\nu_1 \mod I}
\]
due to orthogonality. Since the $l^2$ norm of a probability measure cannot decrease under pushforward, the last line is bounded by:
\[
\ltwo{\nu_1 \mod \mathfrak{m}}=
\ltwo{\xi \mod \mathfrak{m}}^n
\]
The last equality holds because $\nu_1$ is a product measure, and therefore, \[\nu_1 \mod \mathfrak{m}\] is a product measure. The norm of a product measure is the product of the norms of the factors.
\newline
\newline
The remainder of this paper is devoted to the proof of \autoref{lem:centralinequality}.

\marginparr{Wasserstein}

\subsubsection{Proof of \autoref{lem:centralinequality}, using two inequalities from  \S \ref{sec: random matrix inequality} and \S \ref{sec: decomposing measures}}
\label{sec:conditional proof of central inequality}

\newcommand{\Homsubscript}
{
\chi \in Hom(\Mmodule, \omega) / R^{*}
}
\newcommand{\Sursubscript}
{
\chi \in Sur(\Mmodule, \omega) / R^{*}
}

\paragraph{Decomposing measures on modules}
We briefly discuss the decomposition of measures on finite $R$-modules, described
in more detail in \S \ref{sec: decomposing measures}. Recall that, given a signed measure $\nu$ on a
finite module
$\Mmodule$, and a submodule $N \subset \Mmodule$, we denote by $
\projj_N \nu $ the average of $\nu$ over $N$-cosets of $\Mmodule$.
Finally, denote by $\omega$ the dualizing module of $R$.

\begin{lemma}
\label{lem:measuresonmodules}
Given a finite $R$-module $\Mmodule$,
\begin{itemize}
\item Any signed measure $\nu$ on $\Mmodule$
admits a decomposition into orthogonal components parametrized by
$\Homsubscript$:
\[
\nu =  \sum_{\chi} \nu_{\chi}
\]
where
\begin{enumerate}
\item The signed measures $\nu_{\chi}$ are constant on $ker(\chi)$-cosets.
\item $\projj_N \nu_{\chi} =0$ over any N that is not contained in
$ker(\chi)$.
\end{enumerate}
\item We have the upper bound:
\begin{equation}
\label{eqn:uniformestimate}
\lone{\nu_\chi} \leq \sqrt{|\im \chi|}
\end{equation}
\end{itemize}
\end{lemma}
\begin{remark} It follows from 1. and 2. and the definition of $\projj_N \nu $ that:
\[
\projj_N \nu = \sum_{
N \subset ker(\chi)
}
\nu_{\chi}
\]
\end{remark}

\marginparr{Perhaps give the equality number.}

\paragraph{A Reduction}
Now we will apply the preceding decomposition to the signed measure \[ \nu
\defeq \nu_1 -
\nu_0 \] on $R^n$.

\begin{sublemma}
\label{lem:firstreduction}

\begin{equation}
\label{eqn:tvnorm}
\sum_{M} \probP(M=\probMM)
\lone{ \projj_{im(M)}(
\nu_1
-
\nu_0
)
}
\end{equation}
is bounded above by
\begin{equation}
\label{eqn:firstreduction}
\sum_{
\substack{
\chi \in Hom(R^{\rows},\omega)/R^{*}
\\
\chi \neq 0
}
}
\lone{ \nu_{\chi} }
\probP(\chi \, \probMM  = 0)
\end{equation}

\end{sublemma}

\begin{proof}

By (\ref{lem:measuresonmodules})
\begin{equation}
\label{eqn:nuchidecomposition}
\lone
{
\projj_{im(M)} \nu
}
=
\lone{
\sum_{
\substack{
\chi \in Hom(R^n, \omega) /R^{*}
\\
im(M) \subset ker(\chi)
}
}
\nu_{\chi}
}
\end{equation}

Observe that
\[
im(M) \in ker(\chi) \Leftrightarrow
\chi \, M = 0
\]

\newcommand{\chisubscript}{
\substack{
\chi \in Hom(R^{\rows}, \omega)/R^{*}
\\
\chi \, M = 0
}
}

Thus, we can rewrite (\ref{eqn:nuchidecomposition}) as
\[
\lone{
\sum_{
\chisubscript
}
\nu_{\chi}
}
\leq
\sum_{\chisubscript}
\lone{
\nu_{\chi}
}
=
\]
\[
\sum_{
\chi \in Hom(R^{\rows}, \omega)/R^{*}
}
\lone{\nu_\chi}
\One_{
\{
\chi \, M = 0
\}
}
\]

Therefore, (\ref{eqn:centralinequality}) is bounded above by:
\[
\sum_M
\probP \Big( \probM_{\rows,\cols} = M \Big) \sum_{
\chi \in Hom(R^{\rows},\omega)/R^{*}
}
\lone{
\nu_\chi
}
\One_{
\{
\chi \, M = 0
\}
}
\]
\[
=
\sum_{
\chi \in Hom(R^{\cols}, \omega)/R^{*}
}
\lone{\nu_\chi}
\probP( \chi \, \probMM =0)
\]

\newcommand{\notzerosubscript}{
\substack{
\chi \in Hom( R^{\rows}, \omega ) /R^{*}
}
}
\end{proof}

\paragraph{ Proof of \autoref{lem:centralinequality} }
We recall that $\omega$ is the dualizing module. Hence, the correspondance between
submodules of $\omega$ and their annihilating ideals is $1$-to-$1$. Denote by $\omega_I$ the submodule corresponding to $I$.

Hence, we can rewrite (\ref{eqn:firstreduction}) as
\[
\sum_{
I \subset  R
}
\hspace{0.05in}
\sum_{\chi \in Sur(R^{\rows} , \omega_I) / R^{*}}
\lone{\nu_{\chi}}
\probP( \chi\, \probMM =0 )
\]

\autoref{lem:centralinequality} will be proven by showing the
following estimate and summing over all $I$.

\begin{sublemma}
For any ideal $I$ of $R$,
\begin{equation}
\label{eqn:chisumomegai}
\sum_{\chi \in Sur(R^n , \omega_I)/R^{*} }
\lone{\nu_{\chi}} \probP(\chi \, \probM_{\rows, \cols} = 0 ) \leq
\end{equation}
\[
\leq
O\Big(
(1+\epsilon)^{\rows}
\ltwo{\nu \mod I}
\Big)
+
\]
\[
+
O\left(
(1+\epsilon)^{\rows}
\maxbound^{\rows}
\right)
\]
where the implied constants depend on $u$, $\epsilon$, $R$ and the
distribution of $\xi$.
\end{sublemma}

\newcommand{\Maxzero}[1]{
\max
\Big[
#1
-
\left(
\frac{1+\epsilon_0}{|\omega_I|}
\right)^{\cols}
,
0
\Big]
}

\newcommand{\sursubscript}
{
\chi \in Sur(R^{\rows}, \omega_I)/R^{*}
}

\begin{proof}
By \autoref{eqn:uniformestimate}, $|\nu_{\chi}|$ is bounded above by
\[
\sqrt{|\omega_I|} = \sqrt{|R/I|}
\]
\marginparr{Check This!!!!!}
Hence, we can separate the sum (\ref{eqn:chisumomegai}) into two parts:
\[
\sqrt{|R/I|} \sum_{\sursubscript}
\Maxzero{
\probP(\chi \, \probMM = 0)
} +
\]
\[
+ \sum_{\sursubscript}
\lone{\nu_{\chi}} \left( \frac{(1+\epsilon_0)}{|\omega_I|}\right)^{\cols}
\]

It remains to estimate the two terms. The two necessary estimate follow from inequalities that are proven in \S \ref{sec: random matrix inequality} and \S \ref{sec: decomposing measures}:

\paragraph{Applying the bounds from  \S \ref{sec: random matrix inequality} and \S \ref{sec: decomposing measures}}

\begin{itemize}

\item[(A)] By \autoref{thm:fundamentalmomentinequality}, for any $\epsilon'>0$,
\[
\sum_{\sursubscript} \Maxzero{\probP
(\chi \probMM = 0)
}
\leq
\]
\[
\leq
O
\left(
\max
\left[
\frac{1}{
char( R / \mathfrak{m} )
}
,
\linf{ \xi \mod \mathfrak{m} }
\right]^{\rows}
(1+\epsilon')^{\rows}
\right)
\]
where the implied constant depends on $\epsilon'$, $R/I$, $u$
and the distribution of $\xi \mod I$. 
\newline
\newline
Moreover, this inequality remains valid if we replace some of the entries of $\probM_{n,n+u}$ by independent uniformly random variables.

\marginparr{Sort out the epsilons... perhaps replace [cepsilon] by [delta]... }

\item[(B)] By \autoref{thm:maininequality},
\[
\label{eqn:finalequation}
\frac{1}{|\omega_I|^{\rows}}
\sum_{\sursubscript}
\lone{\nu_\chi}=
\]
\begin{equation}
=\frac{1}{|\omega_I|^{\rows}}
\sum_{\sursubscript}
\lone{\nu_\chi}
\leq
\frac{1}{\sqrt{\sfrac{R}{I}^{*}}}
\ltwo{\nu \mod I}
\end{equation}

\end{itemize}
\end{proof}
\marginparr{Perhaps expand on how to go from product measure to product
formula, at some point.}
\marginparr{Rewrite this part of the proof, the push-forward part.}

\newcommand{\bigosum}{\bigoplus}
\newcommand{\osum}{\oplus}
\newcommand{\sgn}{sgn}
\newcommand{\Sset}{S}
\newcommand{\PP}{\mathcal{P}}


\newcommand{\testtest}{test}
\newcommand{\Cc}{\mathcal{C}}
\newcommand{\Com}{\mathbb{C}}
\renewcommand{\Mm}{M}
\newcommand{\cepsilon}{\epsilon}
\newcommand{\Cchi}{m}
\newcommand{\Cchii}{\{ m_i \}}
\newcommand{\cxi}{\Cchi_i \xi_i }
\newcommand{\ppM}{\probM_{n,n+u}}
\renewcommand{\linf}[1]{
\Big| \Big| #1  \Big| \Big|_{l^{\infty} }
}
\renewcommand{\rows}{n}
\renewcommand{\cols}{n+u}
\newcommand{\leqsim}{\lesssim}
\newcommand{\ann}[1]{\text{ann} \, #1}
\newpage
\section{A random matrix inequality}
\label{sec: random matrix inequality}
\paragraph{} \textit{Short Summary.} In this section we prove a general inequality for random
matrices with i.i.d entries over local rings. This inequality bounds a quantity, expression (\ref{eqn:maintermtobound}),
that measures the average number of surjections to a module $M$ from
the cokernel of a random matrix, in a certain large deviation regime. The quantity we bound is related to, but is distinct from the $M$-moment of the cokernel of the random matrix.

\subsection{Introduction}

Let $R$ be a local ring and let $u$ be a fixed integer. Suppose that
$\probM_{n,n+u}$ is a
random matrix whose entries are i.i.d. random variables, denoted by
$\xi_{ij}$. Then \[ coker(\ppM) \] is a random $R$-module. Let $M$ be
a finite $R$-module.
The $M$-\textit{moment} of the random module $coker(\ppM)$ is defined to be:

\begin{equation}
\label{eqn:momentdefinition}
\probE\Big(
\#Sur \big(coker (\ppM), M \big)
\Big).
\end{equation}

The moment method, pioneered in \cite{WoodMoments}, uses the calculation of
$M$-moments for all $M$ as $n\rightarrow \infty$ to determine the
asymptotic distribution of $coker(\ppM)$. This method has by now been used to
establish the universality of cokernels of random matrices in a
wide range of settings. See \cite[\S 2 and \S 3]{WoodICM} for a
fairly recent survey.

\marginparr{In a forthcoming paper, we re-prove and extend some known universality
results using a different approach to random matrix theory, based on the Lindeberg Replacement Technique (\cite{TaoVuTwo} \cite{Lindeberg}). This is a close relative of the approach in \cite{Maples1}. Our approach does not use the
moment method, but nonetheless requires a bound on a quantity,
(\ref{eqn:maintermtobound}), that is
closely connected to (\ref{eqn:momentdefinition}). This paper will focus on
bounding the quantity (\ref{eqn:maintermtobound}), which we define below.}

\paragraph{The quantity that we wish to bound.}

To put our estimate in context, we give a rough outline of how one
would go about calculating the moment (\ref{eqn:momentdefinition}). By
a standard manipulation, (\ref{eqn:momentdefinition}) can be
re-written as a sum over $f \in Sur(R^{\rows},M)$:

\begin{equation}
\label{eqn:rewrittenmoment}
\sum_{f \in Sur(R^{\rows}, M) }
\probP \Big(
    f \left( \ppM \right) = 0
\Big)
\end{equation}
    where $f(\mathcal{M})=0$ means that the evaluation of  $f$ on every
column vector of $\mathcal{M}$ is $0$.

When applying the moment method,
\begin{itemize}
\item[(A)] we wish to establish that for
\textit{most} $f$, and under certain mild conditions on the
distribution of $\xi$,
\begin{equation}
\label{eqn:probf}
\probP(f(\ppM)=0) \approx
\frac{1}{|\Mm|^{\cols}},
\end{equation}
\item[(B)] it is necessary to show that those $f$ for which
(\ref{eqn:probf}) does not hold, have a negligible contribution to the sum
(\ref{eqn:rewrittenmoment}).
\end{itemize}

In our approach, we are interested \textit{solely} in part (B), i.e. the contribution of those $f$ for which
\[
\probP\big(  f(\ppM) =0 \big)
\]
  is atypically large. Specifically, we are interested in the quantity:

\newcommand{\maxzero}[1]{\max \left(
#1 -
\left(
\frac{1+\epsilon_0}{|\Mm|}
\right)^{\cols}
\, , \,0 \,\right)
}

\begin{equation}
\label{eqn:maintermtobound}
\sum_{f \in Sur(R^{\rows}, M) } \maxzero{
\probP
\big(  f(\ppM)=0 \big)
}
\end{equation}

The main theorem we establish is that under the condition that $R$ is
a local ring, and under certain conditions\footnote{We believe these conditions to be necessary.} on $\xi$,
(\ref{eqn:maintermtobound}) decreases exponentially with $\rows$, for
any $\epsilon_0$:

\begin{theorem}
Suppose that $R$ is a local ring. Suppose that the support of $\xi$ is
not concentrated on the translate of a subring or the translate of an
ideal of $R$.
Then, for any finite module $M$, (\ref{eqn:maintermtobound})
decreases exponentially with $\rows$.
\end{theorem}

We give a quantitative statement below.

\begin{definition}
Define $\beta$ such that:
\[
\max\left(
\linf{ \xi \mod \mathfrak{m}  }
,
\frac{1}{ char( R / \mathfrak{m} ) }
\right) = 1 - \beta
\]
where $\mathfrak{m}$ is the maximal ideal of $R$ and $l^{\infty}$
refers to the $l^{\infty}$ (or $sup$) norm of the probability
distribution, that $\xi$ induces on $R/\mathfrak{m}$.
  \end{definition}

\begin{theorem}
\label{thm:fundamentalmomentinequality}
For any positive $\epsilon'$, (\ref{eqn:maintermtobound}) is bounded above by
\begin{equation}
\label{eqn:bigo}
O \Big(
(1-\beta)^{\rows}
(1+\epsilon')^{\rows}
\Big).
\end{equation}
The implied constant depends
on $\epsilon'$, $M$ and $u$, and the minimal non-zero value of 
\begin{equation}
\label{eqn:minnonzero}
\probP(\xi \equiv r \mod \ann{M}) \, r \in R.
\end{equation}
\end{theorem}

\begin{remark}
More succinctly, we could say that the implied constant depends only on $\epsilon'$, $M$ and $u$, and the distribution of $\probP(\xi \equiv r \mod \ann{M})$. We have chosen the formulation above to be precise. In the sequel, we will use the symbol
\[
\alpha
\]
to denote (\ref{eqn:minnonzero}).
\end{remark}

\begin{remark}
Although we will prove \autoref{thm:fundamentalmomentinequality} for all $M$, in our application to random matrices, we will need to know the result only for the modules $M$ that satisfy:
\[
Hom(k,M) \cong k
\]
where $k$ is the residue field of $R$.
\end{remark}

\subsubsection{Approach}
\label{sec:approach}

\newcommand{\sumsubscript}{
\substack{
\Cchi_i \in \Mm
\\
span(\Cchi_i) = \Mm
}
}

First of all, we can use independence to rewrite
(\ref{eqn:maintermtobound}) as
\[
\sum_{\sumsubscript}
\maxzero{
\prod_{j} \probP \left( \sum_i \Cchi_i \xi_{ij} = 0  \right)
} \leq
\]
\[
\leq \sum_{\sumsubscript}
\maxzero{
\prod_{j} \linf{ \sum_i \Cchi_i \xi_{ij} } 
} =
\]
\begin{equation}
\label{eqn:linfsum}
=\sum_{\sumsubscript}
\maxzero{
\linf{ \sum_i \Cchi_i \xi_{i1} }^{\cols}
}
\end{equation}

For notational  convenience, we will denote $\xi_{i1}$ as $\xi_i$.

\begin{remark}
We will henceforth be interested in proving that the sum (\ref{eqn:linfsum})
is bounded above by (\ref{eqn:bigo}).
\autoref{thm:fundamentalmomentinequality} will immediately follow from
this bound.

\end{remark}

\begin{remark}
As we are now only interested in (\ref{eqn:linfsum}), we can make a
simplifying assumption.
\begin{itemize}
\item Note that the value of (\ref{eqn:linfsum}) does not change if we
translate the distribution of $\xi$ by $r \in R$, or if we multiply
the distribution by a unit in $R$. Furthermore, the condition that
$\xi$ is not supported on the translate of a subring also does not
change under these operations.
\item Note that the support of $\xi$ is not concentrated on the
translate of an ideal. Hence it must contain two elements whose
difference is not in $\mathfrak{m}$ and is hence a unit. Therefore,
after translating and multiplying by a unit, we can arrange for the
support of the new random variable to contain $0$ and $1$.
\end{itemize}
Hence, in bounding (\ref{eqn:linfsum}), we can assume, from now on, that
\begin{equation}
\label{eqn:zeroone}
\text{the support of $\xi$ contains $0$ and $1$}.
\end{equation}
\end{remark}

\subsubsection{Strategy}

\begin{definition}
For the rest of this paper, we choose $\cepsilon$ such that
\[
0 < \cepsilon < \epsilon_0
\]
\end{definition}

We estimate (\ref{eqn:linfsum}) by separating the sum into three
components, based on the Fourier transform of the random
variable:
\begin{equation}
\label{eqn:cxi:intro}
\sum_i \cxi.
\end{equation}

\begin{itemize}
\item[\textbf{Type 1}] We will say that $ \{ m_i \}$ is of Type $1$ if the
non-trivial values of the Fourier transform of
\newcommand{\fouriertransform}{
   \mathcal{F} \left[ \sum_i \cxi \right]
}

\[
\sum_i \cxi
\]
are bounded above by $\cepsilon/|\Mm|$ in absolute value.

\item[\textbf{Type 2}]
We say that $\{ \Cchi_i \}$ is of Type $2$ if all the values of the
Fourier transform of
\[
\sum_i \cxi
\]
are \textit{either} $1$ \textit{or} are bounded above by
$\cepsilon/|\Mm|$ in absolute value.

\item[\textbf{Type 3}]
Otherwise, we say that $\{ m_i \}$ is of Type $3$.
\end{itemize}

If $\{ m_i \}$ is of type $j$, for $j=1,2,3$ we will write $\{ m_i \}
\in \Cc_j$.
\newline
\newline
We will call the values of the Fourier transform "small", if they are
bounded above by $\cepsilon/|\Mm|$. We will call the values of
the Fourier transform "large" if they are not "small" and do not have
absolute value equal to $1$. Finally, we remark that because of
$(\ref{eqn:zeroone})$, all the random variables we consider
in subsequent sections contain $0$ in their support. Therefore, if the
Fourier transform of such a variable has absolute value $1$, then it
must equal $1$.
\newline
\newline

In the ensuing sections, we will bound the contributions from $\Cc_1$,
$\Cc_2$ and $\Cc_3$, separately. Indeed, as we shall see, $\Cc_1$ does
not contribute, and the contributions from $\Cc_2$ and $\Cc_3$ are each bounded above by (\ref{eqn:bigo}).

\subsubsection{Outline of this section}
In the first section, we will define the notion of
$\cepsilon$-equidistribution and establish some basic facts. In the
subsequent sections, we analyze successively the contributions from
$\Cc_1$, $\Cc_2$ and $\Cc_3$.



\subsection{The notion of $\cepsilon$-equidistribution}

\begin{definition}
We say that a random variable $\zeta$ is $\cepsilon$-equidistributed
on a finite abelian group $G$ if every non-trivial Fourier coefficient
of $\zeta$ is bounded above by $\cepsilon/|G|$.
\end{definition}

\begin{lemma}
\label{lem:equidist}
If $\zeta$ is $\cepsilon$-equidistributed on $G$, then
\[
\linf{\zeta}=\max_{g \in G} \probP(\zeta =g ) \leq \frac{1+\cepsilon}{|G|}
\]
\end{lemma}

\begin{proof}
This follows readily from the inverse Fourier transform.
\end{proof}

We say that a $G$-valued random variable is
$\cepsilon$-equidistributed on a subgroup $\supp$ of $G$ if $\zeta$ is
supported on $\supp$ and $\cepsilon$-equidistributed on $\supp$. In \S
\ref{sec:cthree}, we will use the following estimate.

\begin{theorem}
\label{thm:equidist}
Suppose that $\zeta_1$ and $\zeta_2$ are $G$-valued random variables.
Now suppose that $\zeta_2$ is $\cepsilon$-equidistributed on a
subgroup $\pi$ of $G$. Then,
\[
\probP(\zeta_1 +\zeta_2 = g) \leq (1+\cepsilon)
\left( \frac{\probP(\zeta_1 \equiv g \mod \supp)}{|\supp|} \right)
\]
\end{theorem}

\begin{lemma}
\label{lem:aux:equidist}
If $\zeta$ and $\zeta'$ are two independent random variables supported
on $\supp$, and $\zeta$ is $\cepsilon$-equidistributed on $\supp$, then
$\zeta+\zeta'$ is $\cepsilon$-equidistributed on $\supp$.
\end{lemma}

\begin{proof} (of \autoref{lem:aux:equidist})
This follows from the multiplicativity of the Fourier transform.
\end{proof}

\begin{proof} (of \autoref{thm:equidist})
We rewrite
\[
\probP(\zeta_1 +\zeta_2 = g) = \probP(\zeta_1 -g +\zeta_2 = 0) =
\]
\[
\probP \big(\zeta_1 - g +\zeta_2 \equiv 0 \mod \supp | \zeta_1 \equiv
g \mod \supp \big)
\probP \big(\zeta_1 \equiv g \mod \supp \big)
\]
After we condition on $\zeta_1 - g \cong 0 \mod \supp$, both $\zeta_1
- g$ and $\zeta_2$ are independent random variables supported on
$\supp$. Hence, we can apply \autoref{lem:aux:equidist} to deduce
\autoref{thm:equidist}.
\end{proof}



\subsection{Sum over $\Cc_1$}

\begin{lemma}
\label{lem:cone}
If $\Cchii \in \Cc_1$,
\[
\linf{\sum_i \cxi}^{\cols} < \left( \frac{1+\epsilon_0}{|\Mm|} \right)^{\cols}
\]
\end{lemma}

\begin{corollary}
If $\Cchii \in \Cc_1$, $\Cchii$ does not contribute to (\ref{eqn:linfsum}).
\end{corollary}

\begin{proof}
Suppose that $\{ m_i \} \in \Cc_1$. Therefore, the non-trivial Fourier
coefficients of
\begin{equation}
\label{eqn:cxisum}
\sum \cxi
\end{equation}
are bounded uniformly in absolute value by $\frac{\cepsilon}{|\Mm|}$.
Then, by \autoref{lem:equidist},
\[
\linf{\sum \cxi}  \leq \frac{1+\cepsilon}{|\Mm|}
\]
and therefore:
\[
\label{eqn:powersumcxi}
\linf{\sum \cxi}^{\cols}
\leq
\left(
\frac{1+\cepsilon}{|\Mm|}
\right)^{\cols}
<
\left(
\frac{1+\epsilon_0}{|\Mm|}
\right)^{\cols}
\]
\end{proof}


\subsection{Sum over $\Cc_2$}
Recall that $\xi_i$ are identically distributed $R$-valued random
variables, whose support is not contained in the translate of any
proper subring of $R$. By (\ref{eqn:zeroone}), we can assume that the
support of $\xi_i$ contains $0$ and $1$.

Recall that $\Cc_2$ is defined to consist of $\rows$-tuples of
elements of $M$, spanning $M$, such that:
\begin{itemize}
\item Every Fourier coefficient of the random variable
\[
\sum_i \Cchi_i \xi_i
\]
is either equal to $1$, or bounded above in absolute value by:
\[
\frac{\cepsilon}{|\Mm|}.
\]
\item At least one non-trivial Fourier coefficient is equal to $1$.
\end{itemize}

We have the following theorem:

\begin{theorem}
\label{thm:ctwo}
\[
\sum_{ \{ m_i \} \in \Cc_2 }
\linf{ \sum_i \Cchi_i \xi_i }^{\cols}
\leqsim
\Big(  \frac{1+\epsilon}{ char(R/\mathfrak{m}) }
\Big)^{\rows}
\]
where the implied constant depends only on $M$ and $u$.
\end{theorem}

\begin{lemma}\
Suppose that $\Cchii \in \Cc_2$. There is a proper additive subgroup
$\supp_{\Cchii}$ of $M$ such that:
\begin{itemize}
\item The random variables $\cxi$ are supported on $\supp_{\Cchii}$.
\item We have the bound
\[
\linf{
\sum_i \cxi
}
\leq
\frac{1+\cepsilon}{ |\supp_{\Cchii}| }
\]
\end{itemize}
\end{lemma}

\begin{proof}
Indeed, if a non-trivial Fourier coefficient is equal to $1$, then the
random variable:
\begin{equation}
\label{eqn:randomvariable}
\sum_i \cxi
\end{equation}
is supported on the kernel of the corresponding homomorphism to
$\mathbb{C}^{*}$. Therefore, the support of (\ref{eqn:randomvariable})
is contained in a proper additive subgroup of $\Mm$. Denote
by $\supp_{\Cchii}$ the smallest additive subgroup that contains the
support of (\ref{eqn:randomvariable}).

\begin{claim}
The Fourier coefficients of (\ref{eqn:randomvariable}), regarded as a
random variable valued in $\supp_{\Cchii}$ are a subset of the Fourier
coefficients of (\ref{eqn:randomvariable}), regarded as a random
variable valued in $M$.
\end{claim}

The claim follows from the fact that $\Q/\Z$ is injective in the
category of abelian groups. Therefore, any homomorphism in
$Hom(\supp,\Com^{*}) \cong Hom(\supp_{\Cchi}, \Q/\Z)$ can be extended to
a homomorphism in $Hom(M,\Com^{*}) \cong Hom(M,\Q/\Z)$.
\newline
\newline

It follows from the claim that all the non-trivial Fourier
coefficients of (\ref{eqn:randomvariable}), regarded as a random
variable valued in $\supp_{\Cchii}$, are either equal to $1$ or bounded
above by $\ssfrac{\cepsilon}{|M|}$. But by the minimality of
$\supp_{\Cchii}$, none of the non-trivial Fourier coefficients can be
equal to $1$. It follows that the random variable (\ref{eqn:randomvariable}) is
$\cepsilon$-equidistributed on $\supp_{\Cchii}$. Therefore, by \autoref{lem:equidist}
\[
\linf{
\sum_i \cxi
}
\leq
\frac{1+\cepsilon}{|\supp_{\Cchii}|}
\]
\end{proof}

\begin{corollary}

\begin{equation}
\label{eqn:corollaryequation}
\sum_{\Cchii \in \Cc_2} \linf{ \sum_i \cxi }^{\cols} \leq
\sum_{\Cchii \in \Cc_2} \left(
\frac{1+\epsilon}{|\supp_{\Cchii}|}
\right)^{\cols}
\end{equation}

\end{corollary}

The right hand side of (\ref{eqn:corollaryequation}) is bounded above by:
\newcommand{\ctwoexpressionone}[1]{
\sum_{
\substack{
\supp \subset M
\\
\supp \neq M
}
}
\left(
\frac{1+\cepsilon}{|\supp|}
\right)^{\cols}
\#\left\{
#1
support\left(
\sum_i \cxi
\right) \in \supp
\right\}
}

\[
\ctwoexpressionone{\Cchii \subset \Cc_2 \Big| }
\leq
\]
\[
\leq
\ctwoexpressionone{\Cchii \Big| span\{ \Cchi_i\}=\Mm; }
\]

\begin{lemma}
Suppose that the $\rows$-tuple $\Cchii$ spans $M$, and
\[
support\left(
\sum_i \cxi
\right) \in \supp .
\]
If $\supp \neq M$, then $\supp$ is not an $R$-module.
\end{lemma}

\begin{proof}
Because the support of $\xi_i$ contains $0$ and $1$, it follows that
the support of
\[
\sum_i \cxi
\]
contains $\Cchi_i$ for all $i$. It follows that $\Cchi_i \in \supp$.
The smallest $R$-module that contains $\Cchi_i$ for all $i$ is
$M$. Hence, if $\supp \neq M$, then $\supp$ is not an
$R$-module.
\end{proof}

\begin{lemma}
\label{lem:notamodule}
Suppose that $S$ is a subset of $R$ that contains $1$ and that is not
contained in any proper subring of $R$. Suppose that $\supp$ is any
additive subgroup of $M$. Then,
\[
\left\{ m \in M  \Big| mS \in \supp   \right\} \in \supp
\]
where equality holds if and only if $\supp$ is an $R$-module.
\end{lemma}

\begin{proof}
The inclusion holds because $1 \in S$. If $\supp$ is an $R$-module,
then we have equality. Conversely, suppose that we have equality. Then,
\[
\supp S = \supp.
\]
Therefore, $S$ is contained in the stabilizer of $\supp$, which is a
subring of $R$. Since $S$ is not contained in any proper subring of
$R$, the stabilizer of $\supp$ must be $R$.
\end{proof}

\newcommand{\chisuppset}[1]{
\left\{ #1 \Big|
support(\cxi) \in \supp
\right\}
}

\begin{corollary}
Suppose that $\supp$ is not an $R$-module. Then,

\[
\#
\chisuppset{\Cchii}
\leq
\Big( \frac{|\supp|}{char(R/\mathfrak{m})} \Big)^{\rows}
\]
\end{corollary}

\begin{proof}
Indeed,
\[
\#\chisuppset{\Cchii} \leq
\#\left\{
m \in \Mm
\Big|
m \xi \in \supp
\right\}^{\rows}
\]
By \autoref{lem:notamodule}, and because $\supp$ is not an $R$-module,
\begin{equation}
\label{eqn:pisubgroup}
\left\{
m \in \Mm \Big|
m \xi \in \supp
\right\}
\end{equation}

is a proper subgroup of $\supp$. If $p=char(R/\mathfrak{m})$,
$\Mm$ is a $p$-group and hence $\supp$ is a $p$-group,
Therefore, (\ref{eqn:pisubgroup}) is bounded above by $|\supp|/p$ and the result follows.
\end{proof}

\paragraph{Proof of \autoref{thm:ctwo}}
We combine the estimates above to get:
\[
\sum_{\Cchi \in \Cc_2} \linf{\cxi}^{\cols}
\leq
\]
\[
\leq
\sum_{
\substack{
\supp \subset M\\
\supp \notin \mathbf{R-mod}
}
}
\#
\chisuppset{\Cchii}
\left(
\frac{1+\cepsilon}{|\supp|}
\right)^{\cols}
\leq
\]
\[
\leq
\sum_{
\substack{
\supp \subset M\\
\supp \notin \mathbf{R-mod}
}
}
\left(
\frac{|\supp|}{|char(R/\mathfrak{m})|}
\right)^{\rows}
\left(
\frac{1+\cepsilon}{|\supp|}
\right)^{\cols} \leqsim
\]
\[
\leqsim \left(
\frac{1+\epsilon}{char(R/\mathfrak{m})}
\right)^{\cols}
\]
where the implied constant depends only on $M$ and on $u$.


\subsection{Sum over $\Cc_3$}
\label{sec:cthree}

Recall that we say that $\{ \Cchi_i\}$\ belongs to $\Cc_3$ if \textit{ at
least one}
Fourier coefficient of the random variable
\[
\sum_i \Cchi_i \xi_i
\]
has absolute value smaller than $1$ but larger than $\cepsilon /|\Mm|$.

\begin{theorem}
\label{thm:maincthreetheorem}
\[
\sum_{\Cchi \in \Cc_3} \linf{ \sum_i \cxi }^{\cols} \leqsim
(1+\epsilon)^{2\rows + u} (1-\beta)^{\cols}
\]
where the implied constant depends on $u$, $\epsilon$, $M$, and the
distribution of \[(\xi \mod \ann{M}).\]
\end{theorem}

\subsubsection{Preliminary Theorem}

This section will be devoted to the proof of the auxiliary
\autoref{thm:auxcthree}. Before stating \autoref{thm:auxcthree}, we
need a lemma.

\begin{lemma}
There exists a positive integer $T$, depending only on $\epsilon$,
$\Mm$ and the distribution of $(\xi \mod \ann{M})$,  such that for any $m
\in \Mm$, any Fourier coefficient of $m \xi$ is either equal to $1$ or
has absolute value less than:
\[
\left( \frac{\cepsilon}{|\Mm|} \right)^{1/T}
\]
\end{lemma}

\begin{proof}
To prove the lemma, it is sufficient to show that every Fourier coefficient of
$m \xi$ is either $1$ or bounded above in absolute value by some
absolute constant, say $\mathbf{C}$, that depends only on the distribution of $(\xi \mod \ann{M})$. But every Fourier coefficient of $m \xi$ occurs as a Fourier coefficient of $(\xi \mod \ann{M})$. The set of Fourier coefficients of $(\xi \mod \ann{M})$ is a finite set that depends only on the
distribution of $(\xi \mod \ann{M})$. Hence, the result follows.

\end{proof}

\newcommand{\extraneousproof}
{
\begin{proof}
Let $\chi \in Hom(\Mm, \Com)$. Let $N$ be the smallest positive
integer that annihilates $M$. Then $\chi$ factors through $Hom(\Mm,
\Z/N\Z)$. Therefore
\begin{equation}
\label{eqn:chimxi}
\chi (m \xi)
\end{equation}
is a random variable valued in the $N^{th}$ roots of unity.
Furthermore, by the definition of $\beta$, either (\ref{eqn:chimxi})
is concentrated at a point, or  (\ref{eqn:chimxi}) takes at least two
different values with probability at least $\beta$.

\begin{sublemma}
\label{sublem:betabound}
Suppose that $\zeta$ is a random variable valued in the $N^{th}$ roots
of unity that takes at least two distinct values with probability at
least $\beta$. Then there exists a constant $C(N,\beta)<1$, depending
only on $N$ and $\beta$, such that
\begin{equation}
|\probE (\zeta)| \leq C(N,\beta)
\end{equation}
\end{sublemma}

\autoref{sublem:betabound} is true by a compactness argument. Indeed,
the set of all probability distributions with the required property is
a compact set. Hence, every continuous function attains it supremum.
Therefore, there exists a random variable with the required property
such that $|\probE(\zeta)|$ is maximal. This maximal value cannot be
equal to $1$ and it depends only on $N$ and $\beta$. This value is
hence the required $C(N,\beta)$.

The set of all probability distributions on the $N^{th}$ roots of
unity that take at least two distinct values with probability $\beta$,
is a compact set. The absolute value of the expectation is a
continuous function on this set; therefore the supremum is attained
for some distribution. Since this supremum is not equal to $1$, it
must be strictly smaller than $1$.

\end{proof}
}

\begin{definition}
Define $\alpha$ to be the smallest non-zero value of $\probP( \xi
\equiv r \mod \text{ann} \, M)$ as $r$ ranges over $R \mod
\text{ann}\, M$.
\end{definition}

\begin{remark}
Although, this is not very important for the proof, we remark that we can compute an explicit upper bound on $\mathbf{C}$, from $\alpha$ and from the minimal value of $e$ such that $p^e M =0$.
\end{remark}

\begin{theorem}
\label{thm:auxcthree}
Given $\{ \Cchi_i \} \in \Cc_3$, there exists an additive subgroup $\supp \in M$
such that
\begin{itemize}
\item We have the inequality:
\[
\linf{
\sum_i \cxi
}
\leq
\frac{(1-\alpha)(1+\cepsilon)}{|\supp|}.
\]
\item When $\supp$ is an $R$-module, we have the stronger inequality:
\[
\linf{
\sum_i \cxi
}
\leq
\frac{(1-\beta)(1+\cepsilon)}{|\supp|}.
\]
\item $\Cchi_i \in \supp$ for all but $T|\Mm|$ indices $i$.
\end{itemize}
$\supp$ is not necessarily unique.
\end{theorem}

\paragraph{Proof of \autoref{thm:auxcthree}} We will need a lemma:

\begin{lemma}
Given any $\Cchi \in Hom(R^{\rows}, M)$, there exists a set of indices
$\mathcal{I}$ of size at most $T |\Mm|$, such that
\begin{enumerate}
\item The Fourier coefficients of the random variable:
\begin{equation}
\label{eqn:sumrv}
\sum_{i \notin \mathcal{I} } \cxi
\end{equation}
are all either "\textit{small}" or equal to $1$.
\item Furthermore, there is an additive subgroup of $M$ that contains
\[
support(\cxi) \hspace{0.5in}\text{for all } i \notin \mathcal{I},
\]
but does not contain
\[
support(\cxi) \hspace{0.5in} \text{ for any $i \in \mathcal{I}$}.
\]
\end{enumerate}
\end{lemma}

\begin{proof}
We proceed iteratively, starting with the empty set. The iteration
step is as follows. Suppose that we have an index set $\mathcal{J}$
that satisfies the second condition of the lemma, e.g. the set
$\emptyset$. Further suppose that some Fourier coefficient of
\[
\sum_{i \notin J} \cxi
\] is \textit{large}. Then the kernel of this Fourier coefficient contains:
\[
support(\cxi)
\]
for all but at most $T$ indices $i$. Adding these indices to
$\mathcal{J}$, we obtain a new index set $\mathcal{J}'$. This index
set has the property that the smallest subgroup of $\Mm$ that contains:
\[
support(\cxi) \hspace{0.5in} \text{for all } i \notin \mathcal{J}'
\]
does not contain
\[
support(\cxi) \hspace{0.5in} \text{for any } i \in \mathcal{J}'.
\]
As the support of
\[
\sum_{i \in \mathcal{J}'} \cxi
\]
is strictly contained inside the support of
\[
\sum_{i \in \mathcal{J}} \cxi,
\]
the iteration must halt after at most $|\Mm|$ steps. Therefore the
final index set has cardinality at most $|\Mm|T$.
\end{proof}

Let $\supp$ be the smallest subgroup that contains the support of:
\begin{equation}
\label{eqn:sumnotini}
\left( \sum_{i \notin \mathcal{I}} \cxi \right).
\end{equation}
The Fourier coefficient of the restriction of (\ref{eqn:sumnotini})
are either equal to $1$ or "small". Therefore,
\[
\sum_{i \notin \mathcal{I}} \cxi
\]
is $\cepsilon$-equidistributed on $\supp$.

\begin{remark}
The non-uniqueness of $\supp$, stated in \autoref{thm:auxcthree} is due to the non-uniqueness of our choice of $\mathcal{I}$.
\end{remark}

\begin{lemma}
\label{lem:betaepsilon}
For any $i \in \mathcal{I}$, the random variable $\cxi \mod \supp$
takes at least two distinct values with non-zero probability.
Furthermore,
\begin{itemize}
\item We have the bound
\[
\linf{\cxi \mod \supp} \leq 1 - \alpha.
\]
\item If $\supp$ an $R$-module, we have the stronger bound
\[
\linf{\cxi \mod \supp} \leq
\linf{\xi \mod \mathfrak{m} }
\leq 1- \beta.
\]
\end{itemize}
\end{lemma}

\begin{proof}
The support of $\xi$ contains $0$. Hence the random variable induced
by $\cxi$ on $\Mm / \supp$ takes the value $0$ with positive
probability. Moreover, the support of $\cxi$ is not contained in
$\supp$. Hence, the random variable induced by $\cxi$ on $\Mm / \supp$
also a takes a non-zero value with positive probability. This proves
the first part of the lemma.
\newline
\newline
The probability that this induced random variable takes any
given value is a sum of terms of the form: $\prob( \xi_i = r)$. Hence
this probability is either $0$ or it is bounded below by $\alpha$.
Therefore, since the induced random variable takes at least two
distinct values with non-zero probability, each non-zero probability
must be at least $\alpha$. This proves the second part of the lemma.
\newline
\newline
Lastly, if $\supp$ is an $R$-module, then
\[
\linf{ \cxi \mod \supp}
=
\max_{m \in \Mm} \probP \big(\cxi \equiv m \mod \supp \big)=
\]
\[
=\max_{r \in R} \probP \big(\xi \equiv r \mod (\supp:\Cchi_i)\big)
\leq
\]
\[
\leq
\max_{r \in R} \probP \big( \xi \equiv r \mod \mathfrak{m} \big)
=
\]
\[
=\linf{\xi \mod \mathfrak{m}} \leq 1-\beta
\]
where we have used the notation $(\supp:\Cchi_i)$ to denote the ideal
\[\{ r \in R| r m_i \in \supp \}\]. This is a proper ideal of $R$, as $\supp$
does not contain the support of $\cxi$.
\end{proof}

\begin{proof} (of \autoref{thm:auxcthree})
We combine \autoref{thm:equidist} and  \autoref{lem:betaepsilon}. 
\begin{itemize} 
\item It follows that:

\[
\linf{ \sum_i \cxi  } \leq \frac{(1+\cepsilon)}{|\supp|} \linf{ \sum_i
\cxi \mod \supp } \leq
\]

\[
\leq \frac{(1+\cepsilon)(1-\alpha)}{|\supp|}
\]

\item When $\supp$ is an $R$-module, we have the better bound:
\[
\linf{ \sum_i \cxi}
\leq
\frac{(1+\cepsilon)}{|\supp|}
\linf{\xi \mod \mathfrak{m} }
\leq \frac{(1+\cepsilon)(1-\beta)}{|\supp|}
\]
\end{itemize}

\newcommand{\suppremark}{
\begin{remark} Some non-canonical choices were made in the iterative
proof of \autoref{lem:indexset}, and therefore $\mathcal{I}$ is not
necessarily unique and hence $\supp$ is not necessarily unique.
\end{remark}
}
\newcommand{\cthreeunnecessaryremark}{
The main idea, however, is that the kernel of a large Fourier coefficient
\textit{does} contain the support of $\cxi$ for all but a uniformly
bounded number of indices. We iteratively discard such indices. The
iterative procedure terminates because, each time that, by discarding
indices, we arrange a Fourier coefficient to be equal to $1$, the new
random variable becomes supported on a smaller subgroup of $M$. At the
end, we are left with a random variable all of whose Fourier
coefficients are either \textit{small} or equal to $1$.
\end{remark}
}

\end{proof}

\subsubsection{Proof of \autoref{thm:maincthreetheorem}}
\label{sec:endofproof}
Now, the proof of \autoref{thm:maincthreetheorem} is analogous to the
proof of \autoref{thm:ctwo}.

\newcommand{\combinatorialset}
{
\left\{ 
\Cchii \in M^{\rows} 
\Big| 
\substack{
support(\cxi) \,\in \, \supp \text{ for}
\\
\text{all but $|\Mm|T$ indices $i$ } 
}
\right\}
}
\noindent
By \autoref{thm:auxcthree}, we can bound
\[
\sum_{\Cchii \in \Cc_3} \linf{\sum_i \cxi}^{\cols}
\]
by
\[
\sum_{
\substack{
\supp \subset \Mm \\
\supp \text{ not an} \\
\text{$R$-module}
}
}
\#\combinatorialset \left( \frac{(1+\cepsilon)(1-\alpha)}{|\supp|}
\right)^{\cols}
\]
\begin{equation}
\label{eqn:finalbound}
+
\end{equation}
\[
\sum_{
\substack{
\supp \subset \Mm \\
\supp \text{ is an} \\
\text{$R$-module}
}
}
\#\combinatorialset \left( \frac{(1+\cepsilon)(1-\beta)}{|\supp|}
\right)^{\cols}
\]

\begin{lemma}
For all $\supp$, we have the bound:
\[
\# \combinatorialset \leqsim
\]
\begin{equation}
\label{eqn:combinatorialbound}
\leqsim (1+\epsilon)^{\rows} |\supp|^{\rows}
\end{equation}
where the implied constant depends on $\epsilon$ and $|\Mm|$ and $T$.
When $\supp$ is an $R$-module, we have the stronger bound:
\[
\# \combinatorialset \leqsim
\]
\begin{equation}
\label{eqn:combinatorialboundtwo}
\leqsim (1+\epsilon)^{\rows}
\left(
\frac{|\supp|}{char(R / \mathfrak{m}) }
\right)^{\rows}.
\end{equation}
\end{lemma}

\begin{proof}
In both cases, we have:
\[
\# \combinatorialset \leq
\]
\[
\leq \# \left\{
\Cchi \in \Mm
\Big|
support(m \xi) \in \supp
\right\}^{\rows - |\Mm|T} |\Mm|^{|\Mm|T}
{\rows \choose |\Mm|T}
\]
Now, by \autoref{lem:notamodule},
\begin{equation}
\label{eqn:suppbound:cthree}
\#\left\{
\Cchi \in  \Mm
\Big|
support(m \xi) \in \supp
\right\}
\end{equation}
is bounded by $|\supp|$ when $\supp$ is an $R$-module, and bounded by
$\frac{|\supp|}{char(R / \mathfrak{m})}$ when $\supp$ is not an
$R$-module.
\newline
\newline
Lastly, we note that $ {\rows \choose |\Mm|T}$ is a polynomial in
$\rows$. In particular, it grows slower than any
exponential. Therefore,
\[
{\rows \choose |\Mm|T} \leqsim (1+\epsilon)^{\rows}
\]
where the implied constant depends on $\epsilon$, $\Mm$ and $T$. Therefore,
\[
\combinatorialset \leqsim
\]
\[
\leqsim |\supp|^{\rows} |\Mm|^{|\Mm|T}
(1+\epsilon)^{\rows} \leqsim
\]
\[
\leqsim |\supp|^{\rows} (1+\epsilon)^{\rows}
\]
where again the implied constant depends on $\epsilon$, $\Mm$ and $T$.
\end{proof}

Finally, substituting (\ref{eqn:combinatorialbound}) into
(\ref{eqn:finalbound}) yields
\[
\sum_{
\{ \Cchi_i \} \in \Cc_3} \linf{ \sum_i \cxi}^{\cols} \leqsim
(1+\cepsilon)^{2n+u}(1-\beta)^{\cols}
\]
where the implied constant depends on $\cepsilon$, $u$, $\Mm$ and $T$.
Recalling that $T$ is determined by $\Mm$, $\cepsilon$ and the
distribution of $(\xi \mod \text{ann} \, \Mm)$, we
deduce \autoref{thm:maincthreetheorem}.
\paragraph{Deduction of \autoref{thm:fundamentalmomentinequality}}
Finally, we choose $\cepsilon$ such that $(1+\cepsilon)^{3} <
(1+\epsilon')$. Combining the three estimates - \autoref{lem:cone},
\autoref{thm:ctwo} and \autoref{thm:maincthreetheorem} - we
prove that (\ref{eqn:linfsum}) is bounded above by (\ref{eqn:bigo}). Hence, we deduce \autoref{thm:fundamentalmomentinequality}.

\subsection{Replacing some of the entries of the matrix by independent uniformly random variables}

The purpose of this section is to show that \autoref{thm:fundamentalmomentinequality} continues to hold when we replace some of the entries of $\probM_{n,n+u}$ by independent uniformly random variables.
\newline
\newline
Let $\bar{\probM}$ be an $n \times n+u$ random matrix whose entries are independent random variables. $\bar{\probM}$ is independent of $\probM_{n,n+u}$. The entries of $\bar{\probM}$ are not necessarily identically distributed.
\newline
\newline
\begin{lemma}
\label{lem:uniformentries}
The sum of 
\begin{equation}
\label{eqn:sumofmatrices}
\maxzero{
\probP
\big(  f(\ppM+\bar{\probM})=0 \big)
}
\end{equation}
over all \[f \in 
Sur(R^{\rows}, M)\] is bounded above by (\ref{eqn:bigo})
\end{lemma}

\begin{proof}
Denote the entries of $\bar{\probM}$ by $\bar{\xi}_{ij}$. We proceed as in \S \ref{sec:approach}. By independence, we can rewrite (\ref{eqn:sumofmatrices}) as:
\[
\maxzero{
\prod_{j} \probP \left( \sum_i \Cchi_i (\xi_{ij}+\bar{\xi}_{ij})= 0  \right)
} \leq
\]
\[
\leq
\maxzero{
\prod_{j} \linf{ \sum_i \Cchi_i \xi_{ij} + \Cchi_i \bar{\xi}_{ij} } 
} 
\leq
\]
\begin{equation}
\label{eqn:linfsumbound}
\leq
\maxzero{
\prod_{j} \linf{ \sum_i \Cchi_i \xi_{ij} } 
} 
\end{equation}
The last inequality holds because the $l^{\infty}$ norm of a sum of two independent random variables is bounded above by the $l^{\infty}$ norm of either variable. Hence the sum of (\ref{eqn:sumofmatrices}) is precisely (\ref{eqn:linfsum}). By the last paragraph of \S \ref{sec:endofproof}, (\ref{eqn:linfsum}) is bounded above by (\ref{eqn:bigo}).
\end{proof}

\begin{Corollary}
\autoref{thm:fundamentalmomentinequality} remains true if we replace some subset of the entries of $\ppM$ by independent uniformly random variables.
\end{Corollary}

\begin{proof}
We apply \autoref{lem:uniformentries}; it suffices to let $\bar{\probM}$ be a matrix some of whose entries are independent uniformly random variables, while the other entries are identically $0$.
\end{proof}

\newpage
\section{Decomposing measures on modules}
\label{sec: decomposing measures}

\paragraph{} \textit{Short Summary.} In this section, we study
measures on a finite $R$-module $M$, where $R$ is a finite local ring. We
show that the vector space of these measures admits an orthogonal
decomposition, whose components are parametrized by homomorphisms from
$M$ to the dualizing module of $R$. This can be regarded as a slight
generalization of the usual Fourier decomposition.

  \subsection{Introduction}

Denote $\Z/m\Z$ as $\Z_m$. Let $G$ be a finite $\Z_m$-module, in other
words, $G$ is a finite abelian group that is $m$-torsion. 

\begin{definition}
Let $\PP(G)$ be the vector space of signed measures on $G$.
\end{definition}
\noindent
By Fourier analysis, there is an orthogonal decomposition of $\PP(G)$
into $1$-dimensional subspaces indexed by elements of \[\Hom(G,
\mathbb{C}^{*}) \cong \Hom(G,\Z_m).\] We can group these subspaces
together into those parametrized by "isomorphic" homomorphisms, i.e.
those homomorphisms that differ by an automorphism of $\Z_m$. This
gives a decomposition of $\PP(G)$ into orthogonal subspaces, indexed by 
\[
\chi  \in \Hom(G,\Z_m) \big/ \Z_m^{*}.
\]
\noindent
These subspaces, which we henceforth denote as $V(\chi)$, can be uniquely
characterized by the following three properties, without any reference
to Fourier analysis:
\begin{itemize}
\item[(A)] The elements of $V(G,\chi)$ are constant on $ker(\chi)$-cosets.
\marginparr{In other words, the pushforward map, induced by any homomorphism in the
equivalence class $\chi$, is injective on $V(G,\chi)$.}
\item[(B)] $V(G,\chi)$ and $V(G,\chi')$ are orthogonal unless $\chi=\chi'$.
\item[(C)] The subspaces $V(G,\chi)$ span $\PP(G)$, i.e.
\[
\PP(G) \cong \bigosum_{\chi} V(G,\chi).
\]
\end{itemize}

\paragraph{The present paper.} In this note, $M$ will be a finite
module over a finite local ring $R$. We will define and study a
decomposition of $\PP(M)$ that shares similar properties:

\begin{theorem}
\label{thm:centraltheorem}
Let $R$ be a finite local ring, let $M$ be a finite $R$-module and
denote by $\omega$ the dualizing module of $R$. Let $\PP(M)$ denote the space of signed measures on $M$. To every 
\begin{equation}
\label{eqn:chidef}
\chi \in \Hom_{R}(M,\omega) \big/ R^{*}
\end{equation}
we can associate a vector space $V(M,\chi) \in \PP(M)$ with the
following properties:
\begin{itemize}
\item[(A)] The elements of $V(M,\chi)$ are constant on $ker(\chi)$-cosets.
\marginparr{The pushforward map induced by any homomorphism in the
equivalence class $\chi$, is injective on $V(M,\chi)$.}
\item[(B)] $V(M,\chi)$ and $V(M,\chi')$ are orthogonal unless $\chi = \chi'$.  
\item[(C)] The vector spaces $V(M,\chi)$ span $\PP(M)$, i.e.
\begin{equation}
\label{eqn:decomposition:first}
\PP(M) \cong \bigosum_{\chi} V(M,\chi).
\end{equation}
\end{itemize}
\end{theorem}

The main practical objective of this paper is to prove
\autoref{thm:centraltheorem} and to establish the useful inequalities
(\ref{eqn:centralinequality}) and (\ref{eqn:centralinequalitytwo}). \autoref{thm:centraltheorem} is a
consequence of \autoref{cor:chidecomposition}. The inequality
(\ref{eqn:centralinequality}), proven in \S
\ref{sec:importantinequality}, is a straightforward consequence of the
Cauchy-Schwarz inequality and the above decomposition. The inequality (\ref{eqn:centralinequalitytwo}), proven in \S \ref{sec:extrasection}, is a simple consequence of the Cauchy-Schwarz inequality. Along the way, we also establish some additional properties of the vector spaces $V(M,\chi)$.

\paragraph{Outline.} We start by defining a decomposition of $\PP(M)$ into orthogonal subspaces $V(M,N)$ where $N$ runs through all submodules of $M$. This part is purely formal. Then we show, using some basic facts in commutative algebra, that $V(M,N)$ is non-zero if and only if $N \cong ker(\chi)$ for some $\chi$ of the form (\ref{eqn:chidef}). The subspaces $V(M,\chi)$ are defined as $V(M,ker(\chi))$. In \S \ref{sec:isotypic}, we study a coarsening of the decomposition (\ref{eqn:decomposition:first}). \S
\ref{sec:importantinequality} is devoted to the inequality (\ref{eqn:centralinequality}).

  \subsection{Preliminaries: measures on a module}
 \renewcommand{\projj}{\text{proj}}
 \newcommand{\ccdot}{\, \cdot \,}
  
Throughout, $R$ will be a finite local ring and $M$ will be a finite $R$-module. As in the introduction, we will denote by $\PP(M)$ the real vector space of signed measures on $M$.
  \begin{definition}
 Define 
 \begin{equation}
 \label{eqn:innerproduct}
 \Big< \ccdot , \ccdot \Big>_{M} : \PP(M) \times \PP(M) \rightarrow \mathbb{R}
 \end{equation}
 to be the Euclidean inner product on $\PP(M)$, regarded as $\mathbb{R}^{(\#M)}$.
  \end{definition}

  \begin{definition}
 For any submodule $N$ of $M$, let $\PP(M,N) \subset \PP(M)$ denote the
 space of signed measures on $M$ that are constant on $N$-cosets.
  \end{definition}
  \noindent
We make the following remarks:
\begin{itemize}
\item The vector space $\PP(M,N)$ is isomorphic to the space of measures on
 $M/N$, i.e.
 \[
 \PP(M,N) \cong \PP(M/N)
 \]
  
 \item If $N_2 \subset N_1$, \[\PP(M,N_1)  \subset  \PP(M,N_2).\]
  
 \item We can also define a map in the opposite direction. Indeed, we can
 take the adjoint of the inclusion map 
 $\PP(M,N_1) \rightarrow \PP(M,N_2)$ with respect to the inner product
 (\ref{eqn:innerproduct}), to get a map that we can denote as
 $\projj_{N_1,N_2}$:
  
 \[
 \text{proj}_{N_1,N_2}: \PP(M,N_2) \rightarrow \PP(M,N_1) .
 \]

 \paragraph{Properties of $\text{proj}_{N_1,N_2}$.}
  
 $\text{proj}_{N_1,N_2}$ can also be defined as follows: given a
 measure in $\PP(M,N_2)$, "average out" this measure over $N_1$. This
 definition does not depend on $N_2$ and it is defined on all of
 $\PP(M)$. Hence we will subsequently simply write $\text{proj}_{N_1}$.
  
 We note the following properties of $\text{proj}$ that follow from the
 definition and the preceding discussion:
  
 \begin{itemize}
 \item The restriction of $\text{proj}_N$ to $\PP(M,N)$ is the identity.
 \item $\projj_{N_1} \projj_{N_2} = \projj_{N_1 +N_2}$.
 \end{itemize}

\end{itemize}

 \subsection{Decomposition of $\PP(M)$ into orthogonal subspaces}
  
 \begin{definition}
 Define
 \[
 V(M,N) \defeq
 \bigcap_{
 \substack {N \subset N' \\ N \neq N' }
 }
 \ker\Big(
 \PP(M,N)
 \xrightarrow{ \text{proj}_{N',N} }
 \PP(M,N') 
 \Big)
 \]
 \end{definition}
  
 \begin{lemma}
 \label{lem:orthogonal}
 $V(M,N_1)$ and $V(M,N_2)$ are orthogonal subspaces of $\PP(M)$
 for $N_1 \neq N_2$.
 \end{lemma}
  
 \begin{proof}
 Suppose $\nu_1 \in V(M,N_1)$ and $\nu_2 \in V(M,N_2)$ and $N_1 \neq
 N_2$. Then we have:
 \[
 \Big< \nu_1, \nu_2 \Big> = 
 \Big< \projj_{N_1} \nu_1 , \nu_2
 \Big> =
 \Big<  \projj_{N_2} \projj_{N_1} \nu_1, \nu_2
 \Big> =
 \]
 \[
 =
 \Big< \projj_{N_1 +N_2} \nu_1, \nu_2
 \Big> =
 \Big<
 \projj_{N_1+N_2} \nu_1 
 ,
 \projj_{N_1 +N_2} \nu_2
 \Big>
 =0
 \]
because either $\nu_1$ or $\nu_2$ lies in the kernel of $\projj_{N_1+N_2}$. 
 \end{proof}

 \newcommand{\paragraphinsertv}{
 $V(M,N)$ is the subspace of $\PP(M,N)$ that is orthogonal to
$\PP(M,N')$ for all $N'$ that strictly contain $N$.
 }
  
 \begin{lemma}
 \begin{equation}
 \label{eqn:redundantdecomposition}
 \PP(M) = \bigosum_{N \subset M } V(M,N)
 \end{equation}
 \end{lemma}
  
 \begin{proof}
 In order to prove the lemma, it is sufficient to show that $\PP(M)$ is
 spanned by the vector spaces $V(M,N)$. We prove the lemma by induction
 on $\#M$.
  
 \paragraph{Base case.} For the base case, assume $\#M=1$. Then the
 statement says that the $1$-dimensional space $\PP(0)$ is spanned by the
 $1$-dimensional space $V(0,0) \cong \PP(0)$, which is true.
  
 \paragraph{Induction step.} Now suppose that $M$ is an $R$-module of
 cardinality $\#M>1$ and suppose that the statement is true for all
 modules of cardinality less than $\#M$.  
  
 In particular, the statement of the lemma is true for all modules
 $M/L$, assuming that $L$ is not trivial. Hence we can assume:
 \[
 \PP(M/L) = \bigosum_{N \subset M/L} V(M/L , N).
 \]
  \noindent
 By the identification,
 \[
 \PP(M/L) \cong \PP(M,L) 
 \]
 it also follows that:
 \begin{equation}
 \label{eqn:protodecomposition}
 \PP(M,L) = \bigosum_{L \subset N \subset M} V(M,N).
 \end{equation}
  \noindent
 We conclude the proof using the following lemma, which is a direct
consequence of the definition of $\projj$ as an adjoint operator:
 \begin{sublemma}
\label{sublem:orthogonal}
 \[ \ker\Big(
 \PP(M) \xrightarrow{\projj_N} \PP(M,N)
 \Big) \] is the orthogonal complement of $\PP(M,N)$ in $\PP(M)$.
 \end{sublemma}
 Now it follows from the claim that $V(M,0)$ is the orthogonal complement
 in $\PP(M)$ of the vector space:
 \begin{equation}
 \label{eqn:span}
 span \left\{ \PP(M,N) \Big| N \neq 0 \right\}
 \end{equation}
 Hence every element in $\PP(M)$ can be expressed as a sum of an
 element of $V(M,0)$ and an element of (\ref{eqn:span}). But, by
 (\ref{eqn:protodecomposition}), every element in (\ref{eqn:span})
lies in the span of the vector spaces $V(M,N)$. Hence $\PP(M)$ is also
spanned by
 the vector spaces $V(M,N)$.
 \end{proof}

  \paragraph{Other properties of $V(M,N)$}
We list some other properties of the vector spaces $V(M,N)$:
\begin{itemize}
\item \paragraphinsertv
\item There is an isomorphism: $V(M,N) \cong V(M/N,0)$.
\item The vector spaces $V(M,N)$ and $\PP(M,N)$ are invariant under
translation by any element of $M$.
\end{itemize}

 \subsection{Fourier modules}
 \label{sec:fouriermodules}

In general, many of the vector spaces $V(M,N)$ are trivial. We are
interested in identifying the non-trivial $V(M,N)$. As $V(M,N) \cong
V(M/N,0)$, this amounts to describing those modules $L$ for which
$V(L,0) \neq 0$.

 \begin{definition}
 We say that a module $L$ is \textit{Fourier} if $V(L,0)$ is not $0$.
 \end{definition}
  
 Hence, we can rewrite (\ref{eqn:redundantdecomposition}) as:
 \begin{equation}
\label{eqn:fourierdecomposition}
 \PP(M)  = \bigosum_{
 \substack{ N \subset M \\
 M/N \text{ is Fourier}.
 }
 }
 V(M,N)
 \end{equation}

In the remainder of \S  \ref{sec:fouriermodules}, we will characterize all Fourier modules over $R$. \autoref{thm:mainfouriercharacterization} describes this characterization.
  
 \begin{lemma}
 \label{lem:fouriersubmodule}
 If $L$ is Fourier, and $L'$ is a sub-module of $L$, then $L'$ is Fourier.
 \end{lemma}
  
 \begin{proof}
 Indeed, suppose that $L$ is Fourier and suppose that $L'$ is a
 submodule of $L$.
  
 \begin{claim}
 There exists a non-zero element $\nu \in V(L,0)$ such that the
 restriction of $\nu$ to $L'$, \[ \nu |_{L'}, \] is non-zero.
 \end{claim}

Indeed, $V(L,0)$ contains a non-zero element because $L$ is
 Fourier. Because $V(L,0)$ is translation invariant, we can translate
this element so that its restriction to $L'$ is non-zero.
\sspace
 But $\nu$ lies in $\ker(\projj_N)$ for all $N$. In particular, $\nu$
 lies in $\ker(\projj_N)$ for all $N \subset L'$. It follows that
$\nu|_{L'}$ is a non-zero element of $V(L',0)$. Therefore, $L'$ is
Fourier.
 \end{proof}

\begin{lemma}
\label{lem:insertedlemma}
If a module $L$ has a unique non-zero minimal submodule, then $L$ is Fourier.
\end{lemma}

\begin{proof}
Denote the minimal non-zero submodule as $N_0$. There exist elements
of $\PP(L)$ that are not constant on $N_0$-cosets. Therefore $\PP(M)$
is not contained in $\PP(L,N_0)$. Therefore, because $N_0$ is minimal,
$\PP(L)$ is not contained in the span of 
\begin{equation}
\label{eqn:span:1}
\PP(L,N) \hspace{0.5in} N \neq 0
\end{equation}
\noindent
Therefore, the orthogonal complement of (\ref{eqn:span:1}) in $\PP(L)$
is non-empty. Hence, by \autoref{sublem:orthogonal}, and the
definition of $V(\ccdot,\ccdot)$, $V(L,0)$ is non-empty and $L$ is
Fourier.

\end{proof}

\begin{remark}
The purpose of \S\S  \ref{sec:fouriermodulesoverk} and \ref{sec:fouriermodulesoverr} will be to show that the converse of \autoref{lem:insertedlemma} is true. We will use that $R$ is a local ring.  
\end{remark}

\begin{mysection}
Therefore there exists a non-zero element that is in the orthogonal
complement of $\PP(M,N_0)$.This element is orthognal to all ...
 Suppose $L$ is not Fourier. Then $\PP(L)$ is spanned by the vector
spaces:
\[
V(M,N) \hspace{0.5in} N\neq 0
\]
It follows that $\PP(L)$ is spanned by the vector spaces:
\[
\PP(M.N) \hspace{0.5in} N \neq 0 
\]
\end{mysection}
 
  \subsubsection{Fourier modules that are powers of $k$}
 \label{sec:fouriermodulesoverk}
 \renewcommand{\dim}{\,\text{dim}\,}

Recall that $k$ denotes the residue field of $R$.
 
 \begin{theorem}
 \label{thm:fouriermodulesoverk}
 A module $k^n$ is a Fourier module if and only if $n=0$ or $n=1$.
 \end{theorem}

 In the rest of this section, we prove \autoref{thm:fouriermodulesoverk}.
 First of all, $0$ is always a Fourier module. For $n > 0$, this theorem
 will be proven by showing \autoref{lem:vkdim} and \autoref{lem:notfourier}.

 \begin{lemma}
\label{lem:vkdim}
 $V(k,0)$ has dimension $\#k -1$. In particular $k$ is Fourier.
 \end{lemma}

 \begin{proof}
 We show the lemma using
 \[
 \PP(k) \cong V(0,0) \osum V(k,0)
 \]
 and comparing dimensions. $V(0,0) \cong \PP(0)$ has dimension $1$.
 $\PP(k)$ has dimension $\#k$. Therefore $V(k,0)$ has dimension $\#k-1$.
 \end{proof}
  
 \begin{lemma}
\label{lem:notfourier}
 $k^n$ is not a Fourier module over $k$, for any $n>1$. 
 \end{lemma}
  
 \begin{proof}
 Let $n>1$. Again we use the decomposition:
 \[
 \PP(k^n) \cong \bigosum_{
 N \subset k^n 
 }
 V(k^n,N)
 \]
 Recall that $V(k^n,N) \cong V(k^n/N,0)$. Comparing dimensions, we find
 \[
 (\#k)^n=\dim \PP(k^n) =\sum_{i=0}^{i=n} \# \{ N \subset k^n | k^n /
N \cong k^i
 \} \dim V(k^i,0)
 \] 
 It is now sufficient to show that
 \begin{equation}
 \label{eqn:sumonetwo}
 (\#k)^n=\sum_{i=0}^{i=1} \# \{ N \subset k^n | k^n /N \cong k^i \}\dim V(k^i,0)
 \end{equation}
 But (\ref{eqn:sumonetwo}) can be rewritten as: 
 \[
 1 + (\#k - 1)  \#\{  N \subset k^n  |  k^n /N \cong k   \}
 \]
 Let $\#Sur(k^n,k)$ denote the number of surjective homomorphisms
from $k^n$ to $k$ and let $\# Aut(k) $ denote the number of
automorphisms of $k$ as an $R$-module. The preceding expression becomes:
 \[
 1+ (\#k -1) \frac{\#Sur(k^n,k)}{\# Aut(k) } = 1+\#Sur(k^n,k) = (\#k)^n  
 \]
 \end{proof}
\noindent  
 This concludes the proof of \autoref{thm:fouriermodulesoverk}.

  \subsubsection{Fourier modules over $R$}
 \label{sec:fouriermodulesoverr}
 Recall that $R$ is a finite local ring with residue field $k$.
 In this section, we will determine all Fourier modules over $R$.
 We will need the fact that every finite local ring has a dualizing module \cite[Proposition 21.2]{Eisenbud}. We denote the dualizing module of $R$ as $\omega$.
  
 \begin{theorem}
 \label{thm:mainfouriercharacterization}
 Every Fourier module over $R$ is a submodule of $\omega$.
 \end{theorem}
  
 \begin{proof}
 The zero module is Fourier. Now suppose $M$ is a non-zero Fourier module.
 By \autoref{lem:fouriersubmodule}, every submodule of $M$ must be a
 Fourier module. In particular
 \[
 Hom(k,M)
 \] 
 is a Fourier module. This module is non-zero because $M$ is non-zero,
 and it is isomorphic to $k^n$ for some $n$. By
\autoref{thm:fouriermodulesoverk}, 
 \begin{equation}
 \label{eqn:simplemodule}
 Hom(k,M) \cong k.
 \end{equation}
 Now, for an $R$-module $L$ denote $Hom(L,\omega)$ as $D(L)$. $\omega$
 is the dualizing module, hence $D( \ccdot)$ is a dualizing functor.
 Therefore, from (\ref{eqn:simplemodule}), we get:
 \begin{equation}
 \label{eqn:kisomorphism}
 Hom(D(M),D(k)) \cong k
 \end{equation}
 But $D(k) \cong k$, hence (\ref{eqn:kisomorphism}) implies
 \[
 D(M) \otimes k \cong k
 \]
 It is now a consequence of Nakayama's lemma that $D(M) \cong R/I$ for some
 ideal $I \subset R$. Therefore, 
 \[
 M \cong D(D(M)) \cong D(R/I) \cong Hom( R/I , \omega ).
 \]
 Therefore $M$ is a submodule of $\omega$.
 \end{proof}

 In the sections that follow, we denote $ Hom(R/I, \omega) $ as $\omega_I$.

 \begin{remark}
 We have shown that every Fourier module must be a submodule of
 $\omega$. The converse is also true. Indeed, $\omega$ has a minimal
 non-zero submodule. Therefore, by \autoref{lem:insertedlemma},
$\omega$ is Fourier, and by \autoref{lem:fouriersubmodule}, all
submodules
 of $\omega$ are Fourier.
 \end{remark}
  \subsection{Properties of $\omega_I$}
 \label{sec:propertiesofomegai}
  
 In this section, we list some standard properties of $\omega_I$:
  
 \begin{itemize}
 \item[(A)] $I \rightarrow \omega_I$ is an inclusion-reversing bijection
 between submodules of $\omega$ and submodules of $R$.
 \item[(B)] $I = ann(\omega_I)$.
 \item[(C)]  $ |\omega_I| = |R/I| $.
 \item[(D)] The inclusion $ R/I \hookrightarrow Hom( \omega_I , \omega_I ) $ is an isomorphism.
 \item[(E)] $\omega_I$ is the dualizing module for the finite ring $R/I$.
 \item[(F)] No two distinct submodules of $\omega$ are isomorphic.
 \end{itemize}
 \noindent
 We will only give a proof for the last statement.
  
 \begin{proof} (of (F))
 The submodules of $\omega$ are in bijection with submodules of $R$.
 Thus, we must prove that if
 \[
 \omega_I \cong  \omega_{I'},
 \] 
 then $I = I'$. But $ \omega_I \cong \omega_{I'}$ implies that $
 ann(\omega_I) \cong ann(\omega_{I'})$ and hence $I=I'$.
 \end{proof}
  
 \subsection{Relation with usual Fourier analysis}
  
 We can recast the decomposition of $\PP(M)$ into the form (\ref{eqn:decomposition:first}), given in the introduction and connect it to the decomposition arising in usual Fourier analysis.

 First of all, we introduce an equivalence relation:
 \begin{definition}
 Suppose $\chi, \chi' \in Hom(M,\omega)$. Then we write $\chi \sim
\chi'$ if and
 only if $im(\chi) \cong im(\chi')$.
 \end{definition}
  
 This allows us to rewrite the decomposition
(\ref{eqn:fourierdecomposition}) as 
 \begin{equation}
 \label{eqn:chidecomposition}
 \PP(M) \cong \bigosum_{
\chi \in Hom(M,\omega)/ \sim
}
V(M,ker(\chi)).
 \end{equation}

 \begin{lemma}
 \label{lem:chieqclass}
 Suppose $\chi, \chi' \in Hom(M,\omega)$ and $\chi \sim \chi'$. Then
there exists $r \in R^{*}$ such that $\chi' = r \chi $.
 \end{lemma}

 \begin{Corollary}
\label{cor:chidecomposition}
 Hence, we can rewrite (\ref{eqn:chidecomposition}) as:
 \begin{equation}
 \label{eqn:chidecompositiontwo}
 \PP(M) \cong \bigosum_{\chi \in Hom(M, \omega)/R^{*}} V(M,\chi)
 \end{equation}
  where we write $V(M,\chi)$ to denote $V(M,ker(\chi))$.
 \end{Corollary}

 \begin{proof} (of \autoref{lem:chieqclass})
 $im(\chi)$ and $im(\chi)$ are submodules of $\omega$. If they are
 isomorphic, then by \ref{sec:propertiesofomegai}.F, they must be the same
 submodule of $\omega$, say $\omega_I$. Hence there exists an
 isomorphism $\sigma$ of $\omega_I$ such that $\chi' = \sigma \chi$.
 But by \ref{sec:propertiesofomegai}.D, the only homomorphisms from $\omega_I$ to
 $\omega_I$ are given by multiplication by $R/I$. Hence, the only
 isomorphisms from $\omega_I$ to $\omega_I$ are given by multiplication
 by $\sfrac{R}{I}^{*}$, or alternatively by multiplication by
$R^{*}$. Therefore $\sigma$ must be of this form.
 \end{proof}

 \begin{remark}
 We note again that $\chi$ is an equivalence class of homomorphisms.
 The number of homomorphisms in the equivalence class is determined by
 $im(\chi)$. By the preceding proof, if $im(\chi) = \omega_I$, then
the number of elements in the equivalence class is $\left|
\sfrac{R}{I}^{*} \right|$.
 \end{remark}

 \paragraph{*The dimension of $V(M,\chi)$}
  
 The results in the remainder of this section are not necessary for
the sequel. We include them for completeness.
\noindent
First, we recall that
\[
V(M,\chi)=
V(M,ker(\chi))\cong
V(M/ker(\chi),0)\cong
V(im(\chi),0)
\]
and $im(\chi)$ is of the form $\omega_I$ for some $I$.
 
 \begin{lemma}
 \[
 \dim V(\omega_I,0) =  |\sfrac{R}{I}^{*}|
 \]
 \end{lemma}
  
 \begin{proof}
 This can be proven by induction on $|\sfrac{R}{I}|$. Firstly, the
 statement holds for the maximal ideal because, by \autoref{lem:vkdim},
 \[
 \dim V(k,0) = \#k -1.
 \]
 For the induction step, we note that 
 \[
 \PP(R/I) = \bigosum_{\chi \in Hom(\omega/I, \omega)/R^{*}}V(\omega,
 ker(\chi)) 
 \]
 Hence,
 \[
 \dim \PP(R/I) =
 \sum_J \frac{
 \#\{ \chi \in Hom(\omega_I , \omega) | im(\chi) = \omega_J\}
 }
 {
 |\sfrac{R}{J}^{*}|
 } 
 \dim V(\omega_J,0)
 \]
 The image of $\omega_I$ is either isomorphic to $\omega_I$ or has
 cardinality strictly smaller than $\omega_I$.
 Using the inductive hypothesis, we can therefore conclude that the sum
 on the right can be rewritten as:
  
 \begin{equation}
 \label{eqn:sumone}
 \frac{\#\{ Hom(\omega_I, \omega) | im(\chi) = \omega_I \}
 }{|\sfrac{R}{I}^{*}|} \dim V(\omega_I,0)+
\end{equation}
\[
 +\sum_{J \neq I} \#\{ Hom(\omega_I, \omega) | im(\chi) = \omega_J \}
 \]
 whereas the left hand side is
 \[
 \#\omega_I=\#Hom(\omega_I,\omega)=
\]
 \begin{equation}
 \label{eqn:sumtwo}
 =\sum_J \# \{ \chi \in Hom(\omega_I, \omega)| im(\chi) = \omega_J\}
 \end{equation}
 Comparing (\ref{eqn:sumone}) and (\ref{eqn:sumtwo}), we find that we
 must have 
 \[
 \dim V(\omega_I,0) = \left| \sfrac{R}{I}^{*} \right|
 \]
 \end{proof}
  
 \paragraph{Comparison with the classical case}
  
 Suppose that we have a $\Z/ p^N$ module $G$. Then, the usual
 Fourier decomposition decomposes $\PP(G)$ into one-dimensional
 components parametrized by $Hom(G, \mathbb{C}^{*}) \cong
 Hom(G,\Z/p^N)$. If we group together the components corresponding to
 homomorphsms that have the same kernel, then we recover the
 decomposition (\ref{eqn:chidecompositiontwo}): 
 \[
 \PP(G) \cong \bigosum_{\chi \in \Hom(G,\Z/p^N) \big/ ( \sfrac{\Z}{p^N}^{*} )}
 V(G, ker(\chi))
 \] 
  \subsection{Isotypic Fourier components}
 \label{sec:isotypic}
  
 \newcommand{\commentone}{Try to define the term Fourier Component at
some earlier stage}
  
 In this section we will be interested in the space spanned by a subset
 of the $V(M,N)$. 
 Namely, let us choose a Fourier module. This module is necessarily of
 the form $\omega_I$. We are interested in explicitly describing
  
 \newcommand{\subdecomposition}
 {
\bigosum_{
 \substack{
 N \subset M
 \\
 M/N \cong \omega_I
 }
 } V(M,N)
}

\begin{equation}
\label{eqn:subdecomposition}
\subdecomposition
\end{equation}
  
 \begin{remark}
 If $M/N \cong \omega_I$, then $M/N$ is annihilated by $I$. Therefore, $IM
 \subset N$. It follows that
\[
\subdecomposition \in \PP(M,IM).
\]
\textit{In fact, this is all we need for the sequel.} In the rest of this
 section, for completeness, we give a more precise description of the
 vector space (\ref{eqn:subdecomposition}). 
 \end{remark}
  
 To formulate our theorem, we need some preliminary definitions:
  
 \paragraph{*The spaces $W(M,IM)$}
As mentioned in the preceding remark, the rest of this section is not
necessary for the sequel.

 \begin{definition}
 Define $\PP(M,IM)$ as before, as the set of measure on $M$ that are
 constant on $IM$-cosets. Let $W(M,IM) \subset W(M,JM)$ be the space of
 signed measures on $\PP(M)$ such that:
 \begin{itemize}
 \item Each element of $W(M,IM)$ is constant on $IM$-cosets.
 \item Each element of $W(M,IM)$ lies in the orthogonal complement of
 $\PP(M,JM)$ for all ideals $J$ that strictly contain $I$.
 \end{itemize} 
 \end{definition}

 \begin{lemma}
 $W(M,IM)$ and $W(M,I'M)$ are orthogonal if $I \neq I'$.
 \end{lemma}
  
 \begin{proof}
 The proof is analogous to the proof of \autoref{lem:orthogonal}.
 Define $\projj_{JM}$ to be the operation that averages a measure on
 $M$ over $JM$-cosets. Suppose that $w \in W(M,IM)$ and $w' \in
 W(M,I'M)$. We have 
 \[
 \Big<w,w'\Big> = 
 \Big< \projj_{IM} w , w'
 \Big> =
 \Big< \projj_{I'M} \projj_{IM} w, w'
 \Big> =
\]
\[
=
 \Big< \projj_{(I+I')M} w , w'
 \Big> =
 \Big< \projj_{(I+I')M} w, \projj_{(I+I')M} w'
 \Big> 
 \] 
 The last expression must be $0$ unless $I=I'$, by the same argument
as in the proof of \autoref{lem:orthogonal}.
 \end{proof}

 \begin{lemma}
 \[
 \PP(M) \cong \bigosum_{I} W(M,IM)
 \]
 \end{lemma}
  
 We now give the main theorem which relates this decomposition to the
 previous one:
  
 \begin{theorem}
 \label{thm:twodecompositions}
 \[
 W(M,IM) \cong 
 \bigosum_{
 \substack{
 N \subset M\\ M/N \cong \omega_I
 }} V(M,N)
 \]
 \end{theorem}
  
 First, we show the following lemma:
 \begin{lemma}
 \label{lem:twodecompositions}
 If $M/N \cong \omega_I$, then 
 \[
 V(M,N) \subset W(M,IM)
 \]
 \end{lemma}
  
 \begin{proof} (of \autoref{lem:twodecompositions})
 By assumption, $M/N \cong \omega_I$. $I$ is the annihilator of $\omega_I$.
 Therefore, $N$ contains $IM$, but does not contain $JM$ for any $J$
 that strictly contains $I$. $V(M,N)$ is contained in $\PP(M,IM)$. It
 remains to prove the following claim:
 \begin{claim}
 Suppose that the ideal $J$ strictly contains $I$. Then $V(M,N)$ is
 orthogonal to
 $\PP(M,JM)$.
 \end{claim} The claim can be deduced from the orthogonal decomposition: 
 \[
 \PP(M,JM) \cong \bigosum_{JM\subset N'} V(M,N')
 \]
 \end{proof}
  
 \paragraph{Deduction of \autoref{thm:twodecompositions} from
 \autoref{lem:twodecompositions}}
 We have shown that 
 \[
 \bigosum_{
 \substack
 {N \subset M
 \\
 M/N \cong \omega_I
 }
 }
 V(M,N)
 \subset
 W(M,IM) 
 \]
 To prove the converse, we proceed by contradiction. Suppose that for
some $I$, the inclusion
 is strict. We take the product over all $I$. It follows that
 \[
 \bigosum_{N\subset M} V(M,N) \cong \PP(M) 
 \]
 is strictly contained in 
 \[
 \bigosum_{I} W(M,IM) \cong \PP(M)
 \]
 This gives a contradiction.
 \subsection{An important inequality}
\label{sec:importantinequality}
\renewcommand{\bar}{\Big| \Big|}
\newcommand{\numodi}{\nu \mod I}
\newcommand{\nunmodi}{\nu_N \mod I}

Suppose that $\nu$ is a measure on $M$. Denote by $\nu_N$ the
projection of $\nu$ on $V(M,N)$. Denote
\[
|\nu_N| \defeq \bar \nu_N \bar_{L^{1}(M)},
\]
the $L^1$ norm of $\nu_N$. In the proof of universality for random matrices, we will need to bound the following quantity:

\newcommand{\expression}{
\frac{1}{|M/IM|} \sum_{
\substack{ N \subset M
\\
M/N \cong \omega_I
}
}
|\nu_N|
}

\[
\expression
\]

\begin{definition}
Denote by $(\nu \mod I)$ the measure induced on $M/IM$ by $\nu$, via
push-forward. 
\end{definition}

\begin{remark}
Alternatively, recall that $\projj_{IM} \nu$ is a measure in
$\PP(M,IM)$. Via the isomorphism $\PP(M,IM) \cong \PP(M/IM)$,
$\projj_{IM} \nu$ defines a measure on $M/IM$. This measure is precisely
$(\nu \mod I)$.
\end{remark}

Define
\[
\bar \ccdot \bar_{L^2(M/IM)} \defeq \Big< \ccdot , \ccdot \Big>_{M/IM},
\]
the usual Euclidean norm.

\begin{theorem}
\label{thm:maininequality}
\begin{equation}
\label{eqn:centralinequality}
\expression
\leq
\frac{1}{
\sqrt{| \sfrac{R}{I}^{*}|}
}
\bar \nu \mod I \bar_{L^{2}(M/IM)}
\end{equation}
\end{theorem}

\begin{proof}
We use the Cauchy-Schwarz inequality two times. $\nu_N$ is contant on
$N$-cosets, therefore, since $M/N \cong \omega_I$, $\nu_N$ must be
constant on $IM$-cosets. Therefore,
\[
\bar \nu_N \bar_{L^1(M)} =\bar \nu_N \mod I \bar_{L^1(M/IM)}
=
\]
\[
=
\Big< \nu_N \mod I \, , \, \text{sgn}(\nu_N \mod I) \Big>_{M/IM}
\leq
\]
\[
\leq
\bar \nu_N \mod I \bar_{L^2(M/IM)} \sqrt{\# M/IM}
\]
by Cauchy-Schwarz. Define
\[
\Sset \defeq \{ N \subset M | M/N \cong \omega_I \} 
\]
The left hand side is
\newcommand{\fractionalterm}
{
\sqrt{
\frac{\#\Sset}{|M/IM|}
}
} 
\[
\frac{1}{|M/IM|} \sum_{N \in \Sset} |\nunmodi|
\leq
\frac{1}{\sqrt{|M/IM|}} \sum_{N \in \Sset} \bar\nunmodi\bar_{L^2(M/IM)}
\leq
\]
\begin{equation}
\label{eqn:lastexpression}
\leq
\frac{
\sqrt{\# \Sset}
}
{\sqrt{|M/IM|}}
\sqrt{
\sum_{N \in \Sset} \bar \nunmodi \bar^2_{L^2(M/IM)}
}
\end{equation}
where the last equality follows again by Cauchy-Schwarz. Since the
$\nu_N$ are orthogonal, we can rewrite the preceding expression as:
\begin{equation}
\sqrt{
\frac{\# \Sset}{|M/IM|}
}
\bar
\sum_{N \in \Sset} \nunmodi
\bar_{L^{2}(M/IM)}
\leq
\end{equation}
\[
\leq
\fractionalterm
\bar
\sum_{IM \in N} \nunmodi
\bar_{L^2(M/IM)}=
\fractionalterm
\bar
\nu \mod I
\bar_{L^2(M/IM)}
\]
Finally, we note that
\[
\fractionalterm \leq \sqrt{
\frac{\#Sur(M,\omega_I)}{ \#Hom(M,\omega_I)   |\sfrac{R}{I}^{*} | }
}
\leq
\frac{1}{|\sfrac{R}{I}^{*}|}
\]
\end{proof}

\begin{remark}
Although this does not substantially improve the bound, we remark that
\[
\sum_{N \in \Sset} \nu_N
\]
is the projection of the measure $\nu$ on $W(M,IM)$. Hence, in \autoref{thm:maininequality}, we can replace $(\nu \mod I)$ by the projection of $(\nu \mod I)$ onto $W(M/IM,0)$.
\end{remark}

 \subsection{A uniform bound on the $L^1$ norm of a Fourier component}
 \label{sec:extrasection}

In \S \ref{sec:conditional proof of central inequality}, we also need another inequality.

\begin{lemma}
\label{lem:secondinequality}
If $\nu$ is any probability measure, then 
\begin{equation}
\label{eqn:centralinequalitytwo}
\bar\nu_{\chi}\bar_{L^1(M)} \leq \sqrt{|im(\chi)|}
\end{equation}
\end{lemma}

\noindent
To show (\ref{eqn:centralinequalitytwo}), we first make the following observation:

\begin{claim}
\[\nu_{\chi} \mod ker(\chi)\] is orthogonal to \[(\nu_{\chi} - \nu) \mod ker(\chi).\]
\end{claim}

\begin{proof} (of claim)
This follows from the construction of $\nu_{\chi}$. Indeed, by construction,   $\nu_{\chi}$ is orthogonal to $(\nu_{\chi} - \projj_{ker(\chi)} \nu)$. Now, by construction $\nu_{\chi} = \projj_{ker(\chi)} \nu_{\chi}$. Therefore, 
\[
\projj_{ker(\chi)} \nu_{\chi}
\]
is orthogonal to 
\[
\projj_{ker(\chi)}(\nu_{\chi}-\nu)
\]
The claim follows.
\end{proof}

\begin{proof}(of \autoref{lem:secondinequality})
\[
\bar\nu_{\chi}\bar_{L^1(M)}=\bar\projj_{ker(\chi)} \nu_{\chi}\bar_{L^1(M)} = \bar \nu_{\chi} \mod ker(\chi) \bar_{L^1(M/ker(\chi))} \leq
\]
\begin{equation}
\label{eqn:cauchyschwarz}
\leq \sqrt{|M/ker(\chi)|}\bar \nu_{\chi} \mod ker(\chi) \bar_{L^2(M/ker(\chi))}
\end{equation}
The last inequality follows by applying the Cauchy-Schwarz inequality. Now the preceding claim shows that 
\[
\bar \nu_{\chi} \mod ker(\chi)\bar_{L^2(M/ker(\chi))} \leq \bar\nu \mod ker(\chi)\bar_{L^2(M/ker(\chi))} 
\]
But the right hand side is the $L^2$ norm of a probability measure, and is therefore bounded above by $1$. Hence (\ref{eqn:cauchyschwarz}) is bounded above by
\[
\sqrt{|M/ker(\chi)|} = \sqrt{|im(\chi)|}.
\]
\end{proof}


\bibliographystyle{alpha}
\bibliography{UniversalityBibliography}

\end{document}